\documentclass[11pt]{article}
\usepackage{amscd}
\usepackage{graphics}
\usepackage{color}
\usepackage{cancel}
\usepackage{stmaryrd}
\usepackage{tikz}
\usetikzlibrary{matrix,arrows}
\usepackage{amsmath}
\usepackage{amsthm}
\usepackage{amssymb}
\usepackage{mathtools}
\usepackage{proof}
\usepackage{verbatim,cmll}
\usepackage{setspace}
\usepackage{lscape}
\usepackage{latexsym,relsize}
\usepackage{amscd}
\usepackage{multicol}
\usepackage{bussproofs}
\usepackage{tikz}
\usepackage[position=top]{subfig}
\usepackage{leqno}
\usepackage{authblk}
\usepackage{marvosym}

\usepackage{mathtools}
\usepackage{enumitem}
\usepackage{hyperref}
\usepackage{thmtools}
\usepackage[capitalise]{cleveref}
\usepackage{semantic} 

\renewcommand{\phi}{\varphi}

\renewcommand{\L}{\mathcal{L}}
\newcommand{\RO}{\mathcal{RO}}
\def\int{{\mathsf{int}}}
\def\cl{{\mathsf{cl}}}

\newcommand{\simpa}{\rightsquigarrow}
\newcommand{\univ}{[\forall]}
\newcommand{\ex}{\langle \exists \rangle}
\newcommand{\biimpl}{\leftrightarrow}
\renewcommand{\ne}{\neg}
\newcommand{\impl}{\rightarrow}

\newcommand{\univin}{\univ^{-1}}
\newcommand{\exin}{\langle \exists \rangle^{-1}}
\newcommand{\nablain}[1]{\nabla^{-#1}}
\newcommand{\Deltain}[1]{\Delta^{-#1}}
\newcommand{\Boxin}{\Box^{-1}}
\newcommand{\Diamondin}{\Diamond^{-1}}

\newcommand{\SSIC}{\mathsf{S^2IC}}
\newcommand{\MSSIC}{\mathsf{MS^2IC}}
\newcommand{\UMSSIC}{{\mathsf{UMS^2IC}}}
\newcommand{\Con}{\mathsf{Con}}
\newcommand{\Comp}{\mathsf{Comp}}
\newcommand{\DeV}{\mathsf{DeV}}
\newcommand{\UMComp}{\mathsf{UMComp}}
\newcommand{\AMCon}{\mathsf{AMCon}}
\newcommand{\UMDeV}{\mathsf{UMDeV}}
\newcommand{\KHaus}{\mathsf{KHaus}}
\newcommand{\MKHaus}{\mathsf{MKHaus}}
\newcommand{\AMDeV}{\mathsf{AMDeV}}
\newcommand{\ZUMComp}{\mathsf{ZUMComp}}
\newcommand{\ZAMDeV}{\mathsf{ZAMDeV}}
\newcommand{\ZUMDeV}{\mathsf{ZUMDeV}}
\newcommand{\DFrm}{\mathsf{DFrm}}

\newenvironment{syst}%
  {\left\lVert\begin{matrix*}[l]}%
  {\end{matrix*}\right.}

\newtheorem{theorem}{Theorem}[section]
\newtheorem{lemma}[theorem]{Lemma}
\newtheorem{proposition}[theorem]{Proposition}
\newtheorem{corollary}[theorem]{Corollary}

\theoremstyle{definition}
\newtheorem{definition}[theorem]{Definition}

\newtheorem{remark}[theorem]{Remark}

\providecommand{\MSC}[1]
{
  \small	
  \noindent 2020 \textit{Mathematics Subject Classification.} #1
}

\providecommand{\keywords}[1]
{
  \small	
  \noindent\textit{Key words and phrases.} #1
}

\AtEndDocument{%
  \medskip
  \begin{tabular}{@{}l@{}}%
    \textsc{Institute for Logic, Language and Computation, University of }\\
    \textsc{Amsterdam, the Netherlands}\\
    \textit{E-mail address}: \texttt{N.Bezhanishvili@uva.nl}\\
    \textsc{Department of Mathematics, University of Milan, Italy}\\
    \textit{E-mail address}: \texttt{luca.carai.uni@gmail.com}\\
    \textsc{Department of Mathematics, University of Milan, Italy}\\
    \textit{E-mail address}: \texttt{silvio.ghilardi@unimi.it}\\
    \textsc{School of Mathematics and Statistics, Taishan University, China}\\
    \textit{E-mail address}: \texttt{zhaozhiguang23@gmail.com}\\
  \end{tabular}}

\title{A calculus for modal compact Hausdorff spaces}
\author{Nick Bezhanishvili, Luca Carai, Silvio Ghilardi, and Zhiguang Zhao}

\date{}

\begin{document}
\maketitle

\begin{abstract}
The symmetric strict implication calculus $\SSIC$ is a modal calculus for compact Hausdorff spaces. This  is established through de Vries duality, linking compact Hausdorff spaces with de Vries algebras---complete Boolean algebras equipped with a special relation.
Modal compact Hausdorff spaces are compact Hausdorff spaces enriched with a continuous relation. These spaces correspond, via modalized de Vries duality, to upper continuous modal de Vries algebras. 

In this paper we  introduce the modal symmetric strict implication calculus $\MSSIC$, which extends $\SSIC$. 
We prove that $\MSSIC$ is strongly sound and complete with respect to upper continuous modal de Vries algebras, thereby providing 
a logical calculus for modal compact Hausdorff spaces. 
We also develop a relational semantics for $\MSSIC$ that we employ to show admissibility of various $\Pi_2$-rules in this system.
\end{abstract}

\MSC{03B45, 06E15, 54E05, 06E25.}


\section{Introduction}\label{Sec:Intro}

Dualities between algebras and topological spaces provide a crucial tool in the study of logics, algebras, and topologies. The 
groundbreaking  work of Stone \cite{St36} established the duality between Boolean algebras and Stone spaces, paving the way for numerous subsequent studies on dualities. Prominent instances of these dualities  include the celebrated Priestley duality between distributive lattices and Priestley spaces
\cite{Pr70,Pr72} and the renowned Esakia duality between Heyting algebras and Esakia spaces \cite{Es74, Esa19}. Another essential addition to these dualities is the \emph{de Vries duality} \cite{deV62}, which states that compact Hausdorff spaces are dual to de Vries algebras---complete Boolean algebras enriched with a binary relation satisfying some specific properties. 
From a logical perspective, de Vries algebras provide an algebraic semantics for the \emph{strict symmetric implication calculus}  $\SSIC$ \cite{BeBeSaVe19,Sa16}. It follows from de Vries duality that compact Hausdorff spaces give a topological semantics for $\SSIC$.

\emph{Modal compact Hausdorff spaces} \cite{BeBeHa15} are the compact Hausdorff generalization of modal spaces
(see, e.g., \cite{CZ97, BRV01, BeBeIe16}). 
These spaces are compact Hausdorff spaces endowed with a relation $R$ satisfying some `continuity' conditions, saying that $R$ is point-closed and that the converse of $R$ maps open sets  to open sets and closed sets to closed sets.  
The modal version of de Vries duality~\cite{BeBeHa15} establishes that modal compact Hausdorff spaces are dually equivalent to lower continuous modal de Vries algebras as well as to upper continuous modal de Vries algebras, which are de Vries algebras equipped with certain operators.
Developing a sound and complete calculus for modal compact Hausdorff spaces was left as an open problem. 

In this paper, we solve  this problem by developing a logical calculus for modal compact Hausdorff spaces. In particular, we introduce the calculus $\MSSIC$ by extending  the strict symmetric implication calculus   $\SSIC$
with  a modal operator $\Box$. Subsequently, we introduce the modal calculus $\UMSSIC$ and prove that  
it is strongly sound and complete with respect to upper continuous modal de Vries algebras. This  is achieved  by adding to $\MSSIC$ specific  $\Pi_2$-rules that express upper continuity. However, this also generates a question whether such non-standard rules are indeed necessary for the axiomatization. 
While it is also possible to obtain a calculus strongly sound and complete with respect to lower continuous modal de Vries algebras, its axiomatization is more involved (see \cref{rem:lower continous calculus}). 
For this reason, we leave the investigation of a calculus for lower continuous modal de Vries algebras to a future work.

Non-standard rules for irreflexivity were first introduced by Gabbay \cite{Ga81}. These rules serve the role of quantifiers in propositional modal logics and have found application in various domains since their inception. They have been utilized in temporal logic \cite{Bu80,GaHo90}, region-based theories of space \cite{BaTiVa07,Va07}, and have played a crucial role in establishing completeness results for modal logic systems featuring non-$\xi$-rules \cite{Ve93}. Notably, the $\Pi_2$-rules, a specific class of such non-standard rules \cite{BeBeSaVe19,BeGhLa20,Sa16}, extend and generalize both Gabbay's irreflexivity rule \cite{Ga81} and Venema's non-$\xi$-rules \cite{Ve93}. 
These rules are referred to as $\Pi_2$-rules because they correspond to semantic conditions expressed by $\forall\exists$-statements.

The $\Pi_2$-rules played a role in the axiomatization of de Vries algebras using the strict symmetric implication calculus   $\SSIC$ \cite{BeBeSaVe19,Sa16}. However, in the same study it was shown  that 
these rules were, in fact, \emph{admissible} in the calculus. This crucial finding indicated that they could be omitted in the axiomatization. 

In this paper, we extend this line of work to the calculus $\MSSIC$. Namely, we show that the $\Pi_2$-rule expressing upper continuity is in fact admissible, establishing equivalence of $\MSSIC$ with $\UMSSIC$. As a result, 
$\MSSIC$ is strongly sound and complete with respect to upper continuous modal de Vries algebras and therefore with respect to 
modal compact Hausdorff spaces. Notably, our admissibility proof deviates from the general methods introduced in~\cite{BeBeSaVe19,Sa16}. Our framework lacks the amalgamation/interpolation properties, essential in those methods. Instead, we obtain the admissibility proof by first developing a relational semantics for $\MSSIC$ and subsequently applying  bisimulation expansions.

The paper is organized as follows. In \cref{Sec:Prelim} we recall some preliminary notions about the symmetric strict implication calculus, modal de Vries algebras, and modal compact Hausdorff spaces. In \cref{Sec:Calculus:LMDV:UMDV} we define the modal strict symmetric calculus $\MSSIC$ and its extension $\UMSSIC$. We then prove that $\UMSSIC$ is strongly sound and complete with respect to upper continuous modal de Vries algebras. In \cref{Sec:Admissibility} we establish the admissibility of various $\Pi_2$-rules in $\MSSIC$. As a consequence, we obtain that $\MSSIC$ and $\UMSSIC$ coincide, and that $\MSSIC$ is also strongly sound and complete with respect to finitely additive modal de Vries algebras. A table listing all the classes of algebras considered in the paper can be found on page~\pageref{table}. The Appendix contains the proof of \cref{prop:L Kripke complete}, which states the Kripke completeness of $\MSSIC$ and provides a first-order characterization of the corresponding class of Kripke frames. 

\section{Preliminaries}\label{Sec:Prelim}

In this section we recall the symmetric strict implication calculus and its algebraic and topological semantics, as well as  the definitions of modal compact Hausdorff spaces and modal de Vries algebras.

\subsection{Symmetric strict implication calculus}\label{subsec:SSIC}

The symmetric strict implication calculus $\SSIC$ \cite{BeBeSaVe19,Sa16} is a deductive system in the language $\L$ that extends the language of classical propositional logic with the binary connective $\rightsquigarrow$ of \emph{strict implication}. We write $\univ \phi$ as an abbreviation of $\top \simpa \phi$. We will use the axiomatization of $\SSIC$ from \cite{BeBeSaVe19}, which differs slightly from the equivalent one given in \cite{Sa16}. 

\begin{definition}
The \emph{symmetric strict implication calculus} $\SSIC$ is the deductive system containing all the substitution instances of the theorems of the classical propositional calculus and of the axioms:
\begin{itemize}
\item[(A1)] $(\bot \simpa \phi) \wedge (\phi \simpa \top)$;
\item[(A2)] $[(\phi \vee \psi) \simpa \chi] \biimpl [(\phi \simpa \chi) \wedge (\psi \simpa \chi)]$;
\item[(A3)] $[\phi \simpa (\psi \wedge \chi)] \biimpl [(\phi \simpa \psi) \wedge (\phi \simpa \chi)]$;
\item[(A4)] $(\phi \simpa \psi) \impl (\phi \impl \psi)$;
\item[(A5)] $(\phi \simpa \psi) \biimpl (\ne \psi \simpa \ne \phi)$;
\item[(A8)] $\univ \phi \impl \univ \univ \phi$;
\item[(A9)] $\ne\univ \phi \impl \univ\neg\univ \phi$;
\item[(A10)] $(\phi \simpa \psi) \biimpl \univ (\phi \simpa \psi)$;
\item[(A11)] $\univ \phi \impl (\neg \univ \phi \simpa \bot)$;
\end{itemize}
and is closed under the inference rules
\begin{itemize}
\item[(MP)] $\inference{\phi \quad \phi\to\psi}{\psi}$;
\item[(R)] $\inference{\phi}{\univ \phi}$.
\end{itemize}
\end{definition}

\begin{definition}
A proof of a formula $\phi$ from a set of formulas $\Gamma$ is a finite sequence $\psi_1, \dots, \psi_n$ of formulas such that $\psi_n=\phi$ and each $\psi_i$ is either in $\Gamma$, an instance of an axiom of $\mathsf{S^2IC}$, obtained from $\psi_j,\psi_k$ with $j,k<i$ by applying (MP), or obtained from $\psi_j$ with $j < i$ by applying (R).
If there exists a proof of $\phi$ from $\Gamma$, then we write $\Gamma \vdash_{\SSIC} \phi$. When $\Gamma = \varnothing$, we simply write $\vdash_{\SSIC} \phi$.
\end{definition}

\begin{remark}\label{rem:rightsquigarrow and nabla}
The connective $\rightsquigarrow$ can be equivalently replaced by a binary modality $\nabla$ defined by $\nabla(\phi,\psi)=\neg \phi \rightsquigarrow \psi$, and the rule (R) by two inference rules $\phi/\nabla(\phi, \psi)$ and $\phi/\nabla(\psi, \phi)$.
It can be shown (see \cite{BeBeSaVe19}) that $\SSIC$ meets the requirements of the definition of a modal system given in \cite[p.~3]{BeCaGhLa22}.
\end{remark}

It is shown in \cite[Sec.~5]{BeBeSaVe19} that $\SSIC$ is sound and complete with respect to the classes of contact, compingent, and de Vries algebras. 
We first recall the notion of contact algebra (see, e.g., \cite[Def.~2.1]{DV06}), which plays an important role in region-based theory of space. We will work with its equivalent formulation in terms of subordinations (see, e.g., \cite[p.~214]{DV06} and \cite[Rem.~4.12]{BeBeSoVe17}).

\begin{definition}
A \emph{contact algebra} is a pair $\mathbf{B} = (B, \prec)$, where $B$ is a Boolean algebra and $\prec$ is a binary relation on $B$ satisfying the following conditions:
\begin{enumerate}
\item[(S1)] $0\prec 0$ and $1\prec 1$;
\item[(S2)] $a\prec b$ and $a\prec c$ imply $a\prec b\land c$;
\item[(S3)] $a\prec c$ and $b\prec c$ imply $a\lor b\prec c$;
\item[(S4)] $a\leq b\prec c\leq d$ implies $a\prec d$;
\item[(S5)] $a\prec b$ implies $a\leq b$;
\item[(S6)] $a\prec b$ implies $\neg b\prec \neg a$.
\end{enumerate}
\end{definition}

Let $X$ be a compact Hausdorff space and $\RO(X)$ the complete Boolean algebra of the regular open subsets of $X$ ordered by inclusion. Equipping $\RO(X)$ with the well-inside relation $\prec$ defined by $U \prec V$ iff $\cl(U) \subseteq V$ yields a contact algebra $(\RO(X), \prec)$. These contact algebras satisfy additional conditions defining what are known as de Vries algebras.

\begin{definition}\label{def:contact compingent de Vries}
A contact algebra $\mathbf{B}=(B, \prec)$ is called a \emph{compingent algebra} if it satisfies the following two additional properties: 
\begin{itemize}
\item[(S7)] $a\prec b$ implies there is $c$ with $a\prec c\prec b$;
\item[(S8)] $a\neq 0$ implies there is $b\neq 0$ with $b\prec a$.
\end{itemize}
A compingent algebra $\mathbf{B}$ is called a \emph{de Vries algebra} if $B$ is a complete Boolean algebra.
\end{definition}

By de Vries duality \cite{deV62}, $\RO$ extends to a dual equivalence between the category of compact Hausdorff spaces and the category of de Vries algebras. In particular, every de Vries algebra is isomorphic to one of the form $(\RO(X), \prec)$ for some compact Hausdorff space $X$.

If $\mathbf{B}$ is a contact algebra, we define a binary operation $\rightsquigarrow$ on $\mathbf{B}$ by setting
\[
a \rightsquigarrow b=
\begin{cases}
1 & \text{if $a \prec b$},\\
0 & \text{otherwise.}
\end{cases}
\]
A \emph{valuation} on a contact algebra $\mathbf{B}$ is a  map that assigns elements of $B$ to the propositional letters of the language $\L$. Each valuation $v$ extends to all formulas in $\L$ by setting $v(\phi \rightsquigarrow \psi)=v(\phi) \rightsquigarrow v(\psi)$ and in the usual way for the classical propositional connectives.
We say that a formula $\phi$ is \emph{valid} in a contact algebra $\mathbf{B}$, and write $\mathbf{B} \vDash \phi$, if $v(\phi)=1$ for all valuations $v$ on $\mathbf{B}$. If $\Gamma$ is a set of formulas, then we write $\mathbf{B} \vDash \Gamma$ if $\mathbf{B} \vDash \gamma$ for every $\gamma \in \Gamma$. If $\mathsf{K}$ is a class of contact algebras, then we say that $\phi$ is a semantic consequence of $\Gamma$ over $\mathsf{K}$, and write $\Gamma \vDash_\mathsf{K} \phi$, when for each $\mathbf{B} \in \mathsf{K}$ and valuation $v$ on $\mathbf{B}$ if $v(\gamma)=1$ for every $\gamma \in \Gamma$, then $v(\phi)=1$.

We denote by $\Con$, $\Comp$, and $\DeV$\label{Con,Comp,DeV} the classes of contact, compingent, and de Vries algebras, respectively. The following theorem states that $\mathsf{S^2IC}$ is strongly sound and complete with respect to all these classes of algebras. 

\begin{theorem}\textup{\cite[Thms.~5.2, 5.8, 5.10]{BeBeSaVe19}}\label{thm:alg completeness of SSIC}
For a set of formulas $\Gamma$ and a formula $\phi$, we have
\[
\Gamma \vdash_{\SSIC} \phi \iff \Gamma \vDash_{\Con} \phi \iff \Gamma \vDash_{\Comp} \phi \iff \Gamma \vDash_{\DeV} \phi.
\]
\end{theorem}

We write $\Gamma \vDash_{\KHaus} \varphi$ to denote that a formula $\varphi$ is a semantic consequence of a set of formulas $\Gamma$ with respect to the class of contact algebras of the form $(\RO(X), \prec)$ for some compact Hausdorff space $X$. The following theorem, which is a consequence of \cref{thm:alg completeness of SSIC} and de Vries duality, states that $\SSIC$ is strongly sound and complete with respect to such class of contact algebras. Thus, $\SSIC$ can be thought of as a logical calculus for compact Hausdorff spaces.

\begin{theorem}\textup{\cite[Thm.~5.10]{BeBeSaVe19}}
For a set of formulas $\Gamma$ and a formula $\phi$, we have
\[
\Gamma \vdash_{\SSIC} \phi \iff \Gamma \vDash_{\KHaus} \phi.
\]
\end{theorem}

\subsection{Modal de Vries algebras and modal compact Hausdorff spaces}

Descriptive frames (aka modal spaces) play an important role in modal logic as they form a category dually equivalent to the category of modal algebras. Modal compact Hausdorff spaces \cite{BeBeHa15} are a generalization of descriptive frames.
If $R$ is a binary relation on a set $X$ and $F,G \subseteq X$, we write
\[
R[F]=\{y : xRy \text{ for some } x \in F\} \quad \text{and} \quad  R^{-1}[G]=\{x : xRy \text{ for some } y \in G\}.
\]
When $x,y \in X$, we write $R[x]$ and $R^{-1}[y]$ instead of $R[\{x\}]$ and $R^{-1}[\{y\}]$, respectively.

\begin{definition}
\mbox{}\begin{enumerate}
\item A binary relation $R$ on a compact Hausdorff space $X$ is said to be \emph{continuous} provided
\begin{enumerate}[label=(\roman*)]
\item $R[x]$ is closed for each $x\in X$;
\item $R^{-1}[F]$ is closed for each closed $F\subseteq X$;
\item $R^{-1}[G]$ is open for each open $G\subseteq X$.
\end{enumerate}
\item We call a pair $(X,R)$ a \emph{modal compact Hausdorff space} if $X$ is a compact Hausdorff space and $R$ a continuous relation on $X$.
\end{enumerate}
\end{definition}

Modal de Vries algebras \cite{BeBeHa15} were introduced as algebraic counterparts of modal compact Hausdorff spaces. For our purposes it is convenient to generalize the definition of a modal operator to the setting of contact algebras.

\begin{definition}
Let $(B, \prec)$ be a contact algebra. We call an operator $\Diamond\colon B\to B$ \emph{de Vries additive} if 
\begin{enumerate}
\item $\Diamond 0=0$;
\item $a_1\prec b_1$ and $a_2\prec b_2$ imply $\Diamond(a_1\lor a_2)\prec(\Diamond b_1\lor\Diamond b_2)$.
\end{enumerate}
A \emph{modal contact algebra} is a triple $(B,\prec,\Diamond)$ where $(B,\prec)$ is a contact algebra and $\Diamond$ is de Vries additive.
A modal contact algebra $(B,\prec,\Diamond)$ is called a \emph{modal compingent algebra} or a \emph{modal de Vries algebra} if $(B, \prec)$ is a compingent algebra or a de Vries algebra, respectively.
\end{definition}

We say that $\Diamond$ is \emph{finitely additive} if it preserves finite joins, and that it is \emph{proximity preserving} is $a\prec b$ implies $\Diamond a\prec \Diamond b$. A de Vries additive operator $\Diamond$ is always proximity preserving but not necessarily finitely additive nor order preserving (see \cite[Ex.~4.9]{BeBeHa15}). The following proposition is a straightforward generalization of \cite[Props.~4.8,~4.10]{BeBeHa15} to operators on contact algebras.

\begin{proposition}\label{prop:fin add iff prox pres}
A finitely additive operator on a contact algebra is de Vries additive iff it is proximity preserving.
\end{proposition}

An important role in this paper is played by the de Vries additive operators that are upper continuous.

\begin{definition}\label{def:upper continuous}
An operator $\Diamond$ on a contact algebra $\mathbf{B}$ is \emph{upper continuous} if for each $a \in B$ the meet of the set $\{\Diamond b: a\prec b\}$ exists and is equal to $\Diamond a$. We call a modal contact algebra with an upper continuous operator an \emph{upper continuous modal contact algebra}. Upper continuous modal contact algebras and upper continuous modal de Vries algebras are defined similarly.
\end{definition}

As an immediate generalization of \cite[Prop.~4.15]{BeBeHa15} to the setting of contact algebras, we obtain:

\begin{proposition}\label{prop:upper cont implies finitely additive}
An upper continuous de Vries additive operator on a contact algebra is order preserving and finitely additive.
\end{proposition}

\begin{definition}
If $(X,R)$ is a modal compact Hausdorff space, then we denote by $\Diamond^U$ the operator on the de Vries algebra $\RO(X)$ defined by
\[
\Diamond^U O = \int (R^{-1}[\cl (O)]).
\]
\end{definition}

\begin{theorem}\textup{\cite[Thm.~5.8]{BeBeHa15}}\label{thm:modal KHaus to upper}
If $(X,R)$ is a modal compact Hausdorff space, then $(\RO(X), \prec, \Diamond^U)$ is an upper continuous modal de Vries algebra.
\end{theorem}

By defining appropriate morphisms, the classes of upper continuous modal de Vries algebras and modal compact Hausdorff spaces become categories, which are dually equivalent to each other (see \cite[Thm.~5.14]{BeBeHa15}). In particular, we have the following representation result, where isomorphisms in the category of upper continuous de Vries algebras are structure preserving bijections (see \cite[Prop.~4.19(3)]{BeBeHa15}).

\begin{theorem}\textup{\cite[Thm.~5.11]{BeBeHa15}}\label{thm:upper to modal KHaus}
Each upper continuous modal de Vries algebra is isomorphic to one of the form $(\RO(X), \prec, \Diamond^U)$ for a modal compact Hausdorff space $(X,R)$. 
\end{theorem}

By \cite[Thm.~4.23]{BeBeHa15}, the category of modal de Vries algebras is equivalent to the category of upper continuous modal de Vries algebras, and hence dually equivalent to the category of modal compact Hausdorff spaces. However, isomorphisms in the category of modal de Vries algebras are not necessarily structure preserving bijections.

We end the section by recalling that (upper continuous) modal contact algebras can be equivalently defined via (lower continuous) de Vries multiplicative operators.

\begin{definition}\label{def:de vries mult and lower continuous}
Let $(B, \prec)$ be a modal contact algebra and $\Box\colon B \to B$.
\begin{enumerate}
\item We call $\Box$ \emph{de Vries multiplicative} if
\begin{enumerate}
\item $\Box 1=1$;
\item $a_1\prec b_1$ and $a_2\prec b_2$ imply $\Box a_1\wedge \Box a_2 \prec\Box (b_1\wedge b_2)$.
\end{enumerate}
\item We call $\Box$ \emph{lower continuous} if $\bigvee\{\Box b: b\prec a\}$ exists and is equal to $\Box a$ for each $a\in B$.
\end{enumerate}
\end{definition} 

As observed in \cite[Rem.~4.11]{BezhanishviliGabelaiaEtAl2019},
there is a bijective correspondence between (upper continuous) de Vries additive operators and (lower continuous) de Vries multiplicative operators:

\begin{proposition}\label{prop:box and diamond}
Let $(B, \prec)$ be a contact algebra and $\Diamond$ a unary operator on $B$. If $\Box \colon B \to B$ is given by $\Box b = \neg \Diamond \neg b$ for every $b \in B$, then
\begin{enumerate}
\item\label{prop:box and diamond:item1} $\Diamond$ is de Vries additive iff $\Box$ is de Vries multiplicative;
\item\label{prop:box and diamond:item2} $\Diamond$ is upper continuous iff $\Box$ is lower continuous.
\end{enumerate}
\end{proposition}

If $(X,R)$ is a modal compact Hausdorff space, then the lower continuous de Vries multiplicative operator on $\RO(X)$ corresponding to $\Diamond^U$ is given by
\begin{align*}
\Box^L O &= \int (\cl (\Box_R O));
\end{align*}
where $\Box_R O=X \setminus R^{-1}[X \setminus O]$ for $O \in \RO(X)$.

\section{A calculus for modal compact Hausdorff spaces}\label{Sec:Calculus:LMDV:UMDV}

In this section we introduce the modal system $\UMSSIC$ and show that it is strongly sound and complete with respect to upper continuous modal de Vries algebras. As a consequence, we will obtain that $\UMSSIC$ is also strongly sound and complete with respect to modal compact Hausdorff spaces. 

We first introduce a fragment of $\UMSSIC$ that we call \emph{modal symmetric strict implication calculus} and denote by $\MSSIC$. The language of $\MSSIC$ is obtained by extending the language of $\SSIC$ with a unary connective $\Box$. We abbreviate $\neg\Box\neg$ with $\Diamond$.

\begin{definition}
Let $\MSSIC$ be the propositional modal system obtained by extending $\SSIC$ with the axioms schemes:
\begin{itemize}
\item[(K)] $\Box(\phi \to \psi) \to (\Box \phi \to \Box \psi)$;
\item[(Add)] $(\phi\rightsquigarrow\psi)\to(\Box\phi\rightsquigarrow\Box\psi)$;
\end{itemize}
and the inference rule:
\begin{itemize}
\item[(N)] $\inference{\phi}{\vphantom{A}\Box\phi}$.
\end{itemize}
\end{definition}

Since $\MSSIC$ proves (K) and is closed under the inference rule (N), \cref{rem:rightsquigarrow and nabla} implies that $\MSSIC$ is also a propositional modal system according to the definition given in \cite[p.~3]{BeCaGhLa22}. For a propositional modal system $\mathcal{S}$,  we say that a unary modality $\univ$ of $\mathcal{S}$ is a \emph{universal modality} if the following formulas are theorems in $\mathcal{S}$, where $\star$ ranges over all modalities in the language of $\mathcal{S}$:
\begin{alignat*}{2}
\univ \phi \impl \phi, \hspace{2cm} &\univ \phi \impl \univ \univ \phi, \\
\phi \impl \univ \ne\univ\ne \phi, \hspace{2cm} &\univ(\phi\impl \psi)\impl (\univ \phi \impl \univ \psi),\\
\omit\rlap{$\bigwedge_i \univ (\phi_i \biimpl \psi_i) \impl (\star[\phi_1, \dots, \phi_n]\biimpl \star[\psi_1, \dots, \psi_n])$.}
\end{alignat*}
For more details on universal modalities see, e.g., \cite{GoPa92} and \cite[Sec.~7.1]{BRV01}. The modality $\univ$ defined as $\univ \phi \coloneq \top \rightsquigarrow \phi$ is a universal modality for $\SSIC$ (see \cite[p.~15]{BeCaGhLa22}). We show that this is still true for $\MSSIC$.

\begin{proposition}\label{prop:MSSIC univ modality}
The calculus $\MSSIC$ has a universal modality $\univ$ given by $\univ \phi = \top \rightsquigarrow \phi$. 
\end{proposition}

\begin{proof}
It is sufficient to show that $\univ (\phi \biimpl \psi) \impl (\Box \phi \biimpl \Box \psi)$ is a theorem of $\MSSIC$. Indeed, since $\MSSIC$ extends $\SSIC$ and $\univ$ is a universal modality for $\SSIC$, the remaining formulas in the definition of universal modality are theorems of $\MSSIC$.
We first prove that $\vdash_{\MSSIC} \univ \chi \to \Box \chi$ for any formula $\chi$. Axiom (Add) yields that $\vdash_{\MSSIC}\univ \chi \to (\Box \top \rightsquigarrow \Box \chi)$. Since $\vdash_{\MSSIC} \Box \top \biimpl \top$, we have that $\vdash_{\MSSIC}\univ \chi \to \univ \Box \chi$. Thus, $\vdash_{\MSSIC} \univ \chi \to \Box \chi$ because $\vdash_{\MSSIC} \univ \Box \chi \to \Box \chi$. We now show that $\vdash_{\MSSIC}\univ (\phi \biimpl \psi) \impl (\Box \phi \biimpl \Box \psi)$. By Axiom (K) and the theorem $\univ \chi \to \Box \chi$, it follows that $\univ (\phi \to \psi) \to (\Box \phi \to \Box \psi)$ and $\univ (\psi \to \phi) \to (\Box \psi \to \Box \phi)$ are theorems of $\MSSIC$. Since $\vdash_{\MSSIC} \univ (\chi_1 \wedge \chi_2) \biimpl (\univ \chi_1 \wedge \univ \chi_2)$ for any pair of formulas $\chi_1, \chi_2$, we obtain that $\vdash_{\MSSIC}\univ (\phi \biimpl \psi) \impl (\Box \phi \biimpl \Box \psi)$.
\end{proof}

Let $\mathcal{S}$ be a propositional modal system with a universal modality $\univ$. We call \emph{$\mathcal{S}$-algebras} the Boolean algebras equipped with operators corresponding to the modalities of $\mathcal{S}$ that validate all the theorems of $\mathcal{S}$. 
A standard Lindenbaum construction yields that $\mathcal{S}$ is sound and complete with respect to the class of $\mathcal{S}$-algebras. 
From~\cite[Thm.~3]{jips:disc93} it follows that $\neg [\forall] \neg x$ is a unary discriminator term in all subdirectly irreducible $\mathcal{S}$-algebras, and so 
$\mathcal{S}$-algebras form a discriminator variety. Then $\mathcal{S}$ is also sound and complete with respect to the class of simple $\mathcal{S}$-algebras, where an $\mathcal{S}$-algebra $B$ is simple iff $\univ a = 1$ or $\univ a = 0$ for each $a \in B$ (see \cite[pp.~240--241]{jips:disc93}).

\begin{remark}\label{rem:correspondence modal contact simple MSSIC-algebras}
An $\SSIC$-algebra $(B, \rightsquigarrow)$ is simple iff $a \rightsquigarrow b$ is $0$ or $1$ for every $a,b \in B$. By \cite[Prop.~3.3]{BeBeSaVe19}, there is a bijection between simple $\SSIC$-algebras and contact algebras.
A simple $\SSIC$-algebra $(B, \rightsquigarrow)$ is associated with the contact algebra ${(B, \prec)}$, where $a \prec b$ iff $a \rightsquigarrow b = 1$. Vice versa, a contact algebra $(B, \prec)$ is associated with the simple $\SSIC$-algebra $(B, \rightsquigarrow)$, where $a \rightsquigarrow b = 1$ if $a \prec b$ and $a \rightsquigarrow b = 0$ otherwise.
Similarly, an $\MSSIC$-algebra $(B, \rightsquigarrow, \Diamond)$ is simple iff $a \rightsquigarrow b$ is $0$ or $1$ for every $a,b \in B$. 
The correspondence above extends to a bijection between simple $\MSSIC$-algebras and contact algebras equipped with a finitely additive and proximity preserving operator, which are exactly the finitely additive modal contact algebras by \cref{prop:fin add iff prox pres}.
\end{remark}

The next theorem states that $\MSSIC$ is sound and complete with respect to the class $\AMCon$\label{AMCon} of finitely additive modal contact algebras. A valuation $v$ on a modal contact algebra extends to all formulas in the language of $\MSSIC$ by setting $v(\Diamond \phi)=\Diamond v(\phi)$. Validity and semantic entailment for modal contact algebras are defined similarly to the case of contact algebras (see \cref{subsec:SSIC}).

\begin{theorem}\label{prop:simple MSSIC}
For a formula $\phi$, we have
\[
\vdash_{\MSSIC} \phi \iff \ \vDash_{\AMCon}\phi.
\]
\end{theorem}

\begin{proof}
By \cref{prop:MSSIC univ modality}, $\MSSIC$ has a universal modality. Therefore, it is complete with respect to the class of simple $\MSSIC$-algebras. The statement then follows from the correspondence between simple $\MSSIC$-algebras and finitely additive modal contact algebras.
\end{proof}

To define $\UMSSIC$ we first need to recall the definition of $\Pi_2$-rules. 

\begin{definition}
A $\Pi_2$-rule is an inference rule of the following shape
\[
\inference{F(\overline{\phi},\overline{p})\to\chi}{G(\overline{\phi})\to\chi}, \tag{$\rho$}
\]
where $F,G$ are formulas, $\overline{\phi}$ is a tuple of formulas, $\chi$ is a formula, and $\overline{p}$ is a tuple of propositional letters which do not occur in $\overline{\phi}$ and $\chi$.

Let $\psi$ and $\psi'$ be formulas. We say that $\psi$ is obtained from $\psi'$ by rule $(\rho)$ if there are some formulas $\overline{\phi}$ and $\chi$ such that $\psi'=F(\overline{\phi},\overline{p})\to\chi$, $\psi=G(\overline{\phi})\to\chi$, with the propositional variables $\overline{p}$ not occurring in $\overline{\phi}$ and $\chi$.
\end{definition}

Note that $\Pi_2$-rules behave differently from standard inference rules because of the requirement of the existence of the propositional variables $\overline{p}$ that must not occur in $\overline{\phi}$ and $\chi$. We refer to the introduction for a short overview of the history of $\Pi_2$-rules in modal logic, which goes back to Gabbay’s irreflexivity rule \cite{Ga81}.

We are now ready to introduce the calculus that we will show is strongly complete with respect to upper continuous modal de Vries algebras.

\begin{definition}\label{def:MSSICU}
The calculus $\UMSSIC$ is obtained by adding the following $\Pi_2$-rules to $\MSSIC$:
\begin{itemize}
\item[($\rho$7)] $\inference{(\phi\rightsquigarrow p)\land(p\rightsquigarrow\psi)\to\chi}{(\phi\rightsquigarrow\psi)\to\chi}$;
\item[($\rho$8)] $\inference{p \land (p\rightsquigarrow\phi)\to\chi}{\phi\to\chi}$;
\item[(UC)] $\inference{\Box p \land (p\rightsquigarrow\phi)\to\chi}{\Box\phi\to\chi}$.
\end{itemize}
\end{definition}

Let $\Gamma$ be a set of formulas and $\phi$ a formula. We write $\Gamma \vdash_\UMSSIC \phi$ if there are $\gamma_1, \dots, \gamma_n \in \Gamma$ such that
\[
\vdash_\UMSSIC \univ (\gamma_1 \wedge \dots \wedge \gamma_n) \to \phi.
\] 

\begin{remark}
As we will see in \cref{prop:Pi2-correspondence}, the $\Pi_2$-rules ($\rho$7) and ($\rho$8) in \cref{def:MSSICU} correspond to Axioms (S7) and (S8) of contact algebras (see \cref{def:contact compingent de Vries}) and the $\Pi_2$-rule (UC) corresponds to upper continuity (see \cref{def:upper continuous}). In order to make this correspondence explicit, we follow the enumeration of \cite{Sa16} instead of the one used in \cite{BeBeSaVe19}, where ($\rho$7) and ($\rho$8) are called ($\rho$6) and ($\rho$7), respectively.
\end{remark}

As their name suggests, $\Pi_2$-rules correspond to certain $\forall \exists$ sentences. We will exploit this correspondence to prove strong completeness of $\UMSSIC$ with respect to upper continuous modal de Vries algebras.

\begin{definition}\label{def:Pi(rho)}
If $\rho$ is the $\Pi_2$-rule
\[
\inference{F(\overline{\phi},\overline{p})\to\chi}{G(\overline{\phi})\to\chi},
\] 
then we denote by $\Pi(\rho)$ the first-order sentence in the language of $\MSSIC$-algebras
\[
\forall \overline{x}, z \left(G(\overline{x}) \nleq z \Rightarrow \exists \overline{y} : F(\overline{x},\overline{y}) \nleq z\right),
\]
where the symbol $\Rightarrow$ stands for the first-order implication.
\end{definition}

The first-order sentences associated to the $\Pi_2$-rules of $\UMSSIC$ are the following: 
\begin{enumerate}[leftmargin=1.25cm]
\item[$\Pi(\rho 7)$:] 
For all $x_1,x_2,z$, if $x_1 \rightsquigarrow x_2 \nleq z$, then there exists $y$ such that\\ $(x_1 \rightsquigarrow y) \wedge (y \rightsquigarrow x_2) \nleq z$.
\item[$\Pi(\rho 8)$:]
For all $x,z$, if $x \nleq z$, then there exists $y$ such that $y \wedge (y \rightsquigarrow x) \nleq z$.
\item[$\Pi(\text{UC})$:]
For all $x,z$, if $\Box x \nleq z$, then there exists $y$ such that $\Box y \wedge (y \rightsquigarrow x) \nleq z$.
\end{enumerate}

\begin{proposition}\label{prop:Pi2-correspondence}
Let $(B, \prec, \Diamond)$ be a finitely additive modal contact algebra and $(B, \rightsquigarrow, \Diamond)$ the corresponding simple $\MSSIC$-algebra \emph{(}see \cref{rem:correspondence modal contact simple MSSIC-algebras}\emph{)}.
\begin{enumerate}
\item\label{prop:Pi2-correspondence:item1} $(B, \rightsquigarrow, \Diamond)$ satisfies $\Pi(\rho 7)$ and $\Pi(\rho 8)$ iff $(B, \prec, \Diamond)$ is a modal compingent algebra.
\item\label{prop:Pi2-correspondence:item2} $(B, \rightsquigarrow, \Diamond)$ satisfies $\Pi(\emph{UC})$ iff $(B, \prec, \Diamond)$ is upper continuous.
\end{enumerate}
\end{proposition}

\begin{proof}
(1) follows from \cite[Lem~6.1]{BeBeSaVe19}.

(2). 
Recall from \cref{prop:box and diamond} that the operator $\Box$ on $B$ is defined as $\neg \Diamond \neg$. The sentence $\Pi(\text{UC})$ holds in $(B, \rightsquigarrow, \Diamond)$ iff for every $x,z \in B$ we have that $\Box x \leq z$ whenever $\Box y \wedge (y \rightsquigarrow x) \leq z$ for each $y \in B$. Observe that $\Box y \wedge (y \rightsquigarrow x) \leq z$ holds iff $y \prec x$ implies $\Box y \le z$.
So, $\Pi(\text{UC})$ is equivalent to the requirement that $\Box x \le z$ whenever $z$ is an upper bound of $\{ \Box y : y \prec x \}$.
Since $\Diamond$ is order preserving, $\Box$ is also order preserving, and hence $\Box x$ is an upper bound of $\{ \Box y : y \prec x \}$. Thus, $(B, \prec, \Diamond)$ satisfies $\Pi(\text{UC})$ iff $\Box x = \bigvee \{ \Box y : y \prec x \}$ for every $x \in B$, which means that $\Box$ is lower continuous. By \cref{prop:box and diamond}\eqref{prop:box and diamond:item2}, $\Box$ is lower continuous iff $\Diamond$ is upper continuous. Therefore, $(B, \rightsquigarrow, \Diamond)$ satisfies $\Pi(\text{UC})$ iff $(B, \prec, \Diamond)$ is upper continuous.
\end{proof}

We now recall one of our main tools to prove  completeness of $\UMSSIC$. If $\mathcal{S}$ is a propositional modal system, then there is a bijective correspondence between propositional formulas in the language of $\mathcal{S}$ and terms in the first-order language of $\mathcal{S}$-algebras. In particular, we can think of a propositional formula $\varphi$ as a term and consider the first-order formula $\varphi=1$. 

\begin{theorem}\textup{\cite[Thm.~5.1]{BeCaGhLa22}}\label{thm:completeness pi2 rules}
Let $\mathcal{S}$ be a propositional modal system with a universal modality and $\Theta$ a set of $\Pi_2$-rules. Denote by $\mathcal{S}+\Theta$ the modal system obtained by adding the $\Pi_2$-rules in $\Theta$ to $\mathcal{S}$ and by $T_\mathcal{S}$ the first-order theory of simple $\mathcal{S}$-algebras such that $0 \neq 1$. Then for every formula $\phi$
\[
T_\mathcal{S} \cup \{ \Pi(\rho) : \rho \in \Theta\} \vDash \phi = 1 \iff \ \vdash_{\mathcal{S}+\Theta} \phi.
\]
\end{theorem}

 The next theorem states that $\UMSSIC$ is sound and complete with respect to the class $\UMComp$\label{UMComp} of upper continuous modal compingent algebras. Note that $\UMComp \subseteq \AMCon$ by \cref{prop:upper cont implies finitely additive}.

\begin{theorem}\label{thm:sound and complete}
For a formula $\phi$, we have
\[
\vdash_{\UMSSIC} \phi \iff \ \vDash_{\UMComp}\phi.
\]
\end{theorem}

\begin{proof}
Since $\MSSIC$ has a universal modality, \cref{thm:completeness pi2 rules} yields that $\UMSSIC$ is sound and complete with respect to the class of simple $\MSSIC$-algebras satisfying $\Pi(\rho 7)$, $\Pi(\rho 8)$, and $\Pi(\text{UC})$. \cref{prop:simple MSSIC,prop:Pi2-correspondence} imply that these algebras correspond to the upper continuous modal compingent algebras under the bijection described in \cref{rem:correspondence modal contact simple MSSIC-algebras}.
\end{proof}

The following theorem states that $\UMSSIC$ is also strongly sound and complete with respect to the class $\UMComp$ of upper continuous modal compingent algebras. 

\begin{theorem}\label{thm:strong sound and complete UMC}
For a set of formulas $\Gamma$ and a formula $\phi$, we have
\[
\Gamma\vdash_\UMSSIC\phi \iff \Gamma\vDash_{\UMComp}\phi.
\]
\end{theorem}

\begin{proof}
We first prove the left-to-right implication. Suppose that $\Gamma\vdash_\UMSSIC\phi$. Then there is a finite subset $\Gamma_0 \subseteq \Gamma$ such that $\vdash_\UMSSIC \univ \bigwedge \Gamma_0 \to \varphi$. Thus, \cref{thm:sound and complete} yields that $(B, \prec, \Diamond) \vDash \univ \bigwedge \Gamma_0 \to \varphi$ for any $(B, \prec,\Diamond) \in \UMComp$. If $v$ is a valuation on $B \in \UMComp$ such that $v(\gamma)=1$ for every $\gamma \in \Gamma$, then $v( \bigwedge \Gamma_0)=1$, and hence $v(\univ\bigwedge \Gamma_0)=1$. Since $(B, \prec, \Diamond) \vDash \univ \bigwedge \Gamma_0 \to \varphi$, it follows that $v(\varphi)=1$. This shows that $\Gamma\vDash_{\UMComp}\phi$.

To prove the right-to-left implication, assume that $\Gamma\nvdash_\UMSSIC\phi$. 
Let $\mathcal{ML}^{+}$ be the first-order language of $\MSSIC$-algebras enriched with a set of constants $\{ c_p \}$, where $p$ ranges over all the propositional variables. The class of the simple $\MSSIC$-algebras corresponding to upper continuous modal compingent algebras is an elementary class. Indeed, it is possible to express that $\Diamond b$ is the join of $\{ \Diamond a : b \prec a\}$ for every element $b$ with a first-order sentence in $\mathcal{ML}^{+}$. We denote by $\mathcal{T}$ the elementary theory of this class. If $\psi$ is a formula in the language of $\MSSIC$, then we denote by $\psi'$ the term in $\mathcal{ML}^{+}$ obtained by replacing each propositional letter $p$ in $\psi$ with $c_p$.
We consider the set of sentences $\Sigma \coloneq \mathcal{T} \cup \bigcup \{ \gamma'=1 : \gamma \in \Gamma \} \cup \{ \phi' \neq 1\}$. 
Since $\Gamma\nvdash_\UMSSIC\phi$, for each finite subset $\Gamma_0$ of $\Gamma$ there is $\mathbf{B} \in \mathsf{UMComp}$ such that $\mathbf{B} \nvDash \univ \bigwedge \Gamma_0 \to \varphi$, and hence there is a valuation $v$ on $\mathbf{B}$ such that $v(\gamma)=1$ for every $\gamma \in \Gamma_0$ and $v(\varphi) \neq 1$.
Each model of $\mathcal{T}$ corresponds to an upper continuous modal contact algebra $\mathbf{B}$ together with a valuation on $B$, where the valuation maps a propositional variable $p$ to the interpretation of the constant $c_p$. Thus, every finite subset of $\Sigma$ is satisfiable. By the compactness theorem, there is $\mathbf{B} \in \mathsf{UMComp}$ together with a valuation $v$ such that $v(\gamma)=1$ for every $\gamma \in \Gamma$ and $v(\phi) \neq 1$. This shows that $\Gamma\nvDash_{\mathsf{UMComp}}\phi$.
\end{proof}

It remains to prove completeness of $\UMSSIC$ with respect to upper continuous modal de Vries algebras. We employ a generalization to the modal setting of the MacNeille completions of compingent algebras introduced in \cite[Def.~5.1.2]{Sa16}.
If $B$ is a Boolean algebra, we denote by $\overline{B}$ its MacNeille completion, and identify $B$ with the corresponding Boolean subalgebra of $\overline{B}$. We will use Roman letters to denote the elements of $B$ and Greek letters for the elements of $\overline{B}$. 

\begin{definition}\label{definition:MN:Completion}
Let $\mathbf{B}=(B,\prec,\Diamond)$ be an upper continuous modal compingent algebra. We define $\prec$ and $\Diamond$ on $\overline{B}$ as follows:
\[
\alpha\prec\beta\mbox{ iff there exist }a,b\in B\mbox{ such that }\alpha\leq a\prec b\leq\beta,
\]
\[
\Diamond\alpha=\bigwedge\{\Diamond a: a \in B, \ \alpha\leq a\}.
\]
We call $\overline{\mathbf{B}}=(\overline{B},\prec,\Diamond)$ the \emph{MacNeille completion} of $\mathbf{B}$.
\end{definition}

\begin{lemma}\label{Lemma:MN:Completion}
The MacNeille completion of an upper continuous modal compingent algebra is an upper continuous modal de Vries algebra, and the inclusion map of $\mathbf{B}$ into $\overline{\mathbf{B}}$ preserves and reflects $\prec$ and commutes with $\Diamond$.
\end{lemma}

\begin{proof}
Let $\mathbf{B}=(B, \prec, \Diamond)$ be an upper continuous modal compingent algebra. By \cite[Rem.~5.11]{BeBeSaVe19}, $(\overline{B},\prec)$ is a de Vries algebra. It is an immediate consequence of the definitions of $\prec$ and $\Diamond$ on $\overline{B}$ that their restrictions to $B$ coincide with $\prec$ and $\Diamond$ on $B$, so the inclusion map preserves and reflects $\prec$ and commutes with $\Diamond$.  It then remains to show that $\Diamond$ on $\overline{B}$ is de Vries additive and upper continuous.

Since $0 \in B$ and $\Diamond$ is de Vries additive on $B$, we have $\Diamond 0 = 0$. Suppose that $\alpha_1 \prec \beta_1$ and $\alpha_2 \prec \beta_2$ with $\alpha_1, \alpha_2, \beta_1, \beta_2 \in \overline{B}$. By definition of $\prec$ in $\overline{B}$, there are $a_1,a_2,b_1,b_2 \in B$ such that $\alpha_1 \le a_1 \prec b_1 \le \beta_1$ and $\alpha_2 \le a_2 \prec b_2 \le \beta_2$. Since $\overline{B}$ is a contact algebra, $\alpha_1 \vee \alpha_2 \le a_1 \vee a_2$. It follows from its definition that $\Diamond$ on $\overline{B}$ is order preserving. So, $\Diamond(\alpha_1 \vee \alpha_2) \le \Diamond(a_1 \vee a_2)$, $\Diamond b_1 \le \Diamond\beta_1$, and $\Diamond b_2 \le \Diamond\beta_2$. Since $a_1 \prec b_1$, $a_2 \prec b_2$, and $\Diamond$ is de Vries additive on $B$, we have $\Diamond(a_1 \vee a_2) \prec \Diamond b_1 \vee \Diamond b_2$. Thus,
\[
\Diamond(\alpha_1 \vee \alpha_2) \le \Diamond(a_1 \vee a_2) \prec \Diamond b_1 \vee \Diamond b_2 \le  \Diamond\beta_1 \vee \Diamond\beta_2,
\]
which yields $\Diamond(\alpha_1 \vee \alpha_2) \prec  \Diamond\beta_1 \vee \Diamond\beta_2$. This establishes that $\Diamond$ is de Vries additive on $\overline{B}$.

We now show that $\Diamond$ is upper continuous on $\overline{B}$, which amounts to proving that $\Diamond \alpha$ is the greatest lower bound of $\{\Diamond \beta : \alpha \prec \beta \}$ for each $\alpha \in \overline{B}$. Since $\Diamond$ is order preserving and $\alpha \prec \beta$ implies $\alpha \le \beta$, we have that $\Diamond \alpha$ is a lower bound of $\{\Diamond \beta : \alpha \prec \beta \}$. To show that it is the greatest one, it is sufficient to show that if $\gamma \nleq \Diamond \alpha$, then there is $b \in B$ such that $\alpha \prec b$ and $\gamma \nleq \Diamond b$. Assume that $\gamma \nleq \Diamond \alpha$. By the definition of $\Diamond \alpha$, there is $a \in B$ such that $\alpha \le a$ and $\gamma \nleq \Diamond a$. Since $\Diamond$ is upper continuous on $B$ and the MacNeille completion preserves all existing meets, $\Diamond a = \bigwedge \{ \Diamond b : a \prec b\}$ holds in $\overline{B}$. Thus, there is $b \in B$ such that $a \prec b$ and $\gamma \nleq \Diamond b$. This shows that $\Diamond \alpha= \bigwedge\{\Diamond \beta : \alpha \prec \beta \}$. Thus, $\Diamond$ is upper continuous on $\overline{B}$.
\end{proof}

The following theorem, which is the main result of the section, establishes strong completeness of $\UMSSIC$ with respect to the class $\UMDeV$\label{UMDeV} of upper continuous modal de Vries algebras.

\begin{theorem}\label{thm:completeness upper modal de vries}
For a set of formulas $\Gamma$ and a formula $\phi$, we have
\[
\Gamma\vdash_{\UMSSIC}\phi \iff \Gamma\vDash_{\UMDeV}\phi.
\]
\end{theorem}

\begin{proof}
We first show the left-to-right implication. Assume that $\Gamma\vdash_{\UMSSIC}\phi$. By \cref{thm:strong sound and complete UMC}, $\UMSSIC$ is strongly sound with respect to the class of upper continuous modal compingent algebras $\UMComp$. Since $\UMDeV \subseteq \UMComp$, it follows that $\Gamma\vDash_{\UMDeV}\phi$. To show the other implication, assume that $\Gamma\nvdash_{\UMSSIC}\phi$. By \cref{thm:strong sound and complete UMC}, there is an upper continuous modal compingent algebra $\mathbf{B}$ and a valuation $v$ on $B$ such that $v(\gamma)=1$ for every $\gamma \in \Gamma$ and $v(\phi) \neq 1$. By \cref{Lemma:MN:Completion}, the inclusion of $\mathbf{B}$ into $\overline{\mathbf{B}}$ preserves and reflects $\prec$ and commutes with $\Diamond$. Thus, $v$ can be thought of as a valuation on $\overline{\mathbf{B}}$, which is a member of $\UMDeV$ by \cref{Lemma:MN:Completion}. It follows that $\Gamma\nvDash_{\UMDeV}\phi$.
\end{proof}

We write $\Gamma\vDash_{\MKHaus}\varphi$ to denote that a formula $\varphi$ is a semantic consequence of a set of formulas $\Gamma$ with respect to the class of modal contact algebras of the form $(\RO(X), \prec, \Diamond^U)$ for some modal compact Hausdorff space $(X,R)$. As a consequence of \cref{thm:modal KHaus to upper,thm:upper to modal KHaus,thm:completeness upper modal de vries}, we obtain the following corollary stating that $\UMSSIC$ is strongly sound and complete with respect to such a class of upper continuous modal contact algebras.

\begin{corollary}\label{cor:MSSICU completeness MKHaus}
For a set of formulas $\Gamma$ and a formula $\phi$, we have
\[
\Gamma\vdash_{\UMSSIC}\phi \iff \Gamma\vDash_{\MKHaus}\phi.
\]
\end{corollary}

\section{Admissibility of $\Pi_2$-rules in $\MSSIC$}\label{Sec:Admissibility}

In this section we prove admissibility in $\MSSIC$ of several $\Pi_2$-rules by utilizing relational semantics. In particular, we show that the three $\Pi_2$-rules in the definition of $\UMSSIC$ are admissible in $\MSSIC$, and hence $\UMSSIC$ and $\MSSIC$ coincide.

Formulas in the language of $\MSSIC$ can be interpreted in Kripke frames $(X,T,S)$ where $T$ is a ternary relation and $S$ is a binary relation. Let $v \colon \text{Prop} \to \wp(X)$ be a valuation. The Boolean connectives are interpreted in the standard way and for $x \in X$ we set
\begin{align*}
x \vDash_v \phi \rightsquigarrow \psi \quad & \text{iff} \quad \forall y,z \in X \, (Txyz \text{ and } y \vDash_v \phi \text{ imply } z \vDash_v \psi )\\
x \vDash_v \Diamond \phi \quad & \text{iff} \quad \exists y \in X \, (xSy \text{ and } y \vDash_v \phi).
\end{align*}
For a formula $\phi$, we write $v(\phi)=\{ x \in X : x \vDash_v \phi \}$ and say that $\phi$ is \emph{valid} in the frame $(X,T,S)$ if $v(\phi)=X$. We say that a Kripke frame is an \emph{$\MSSIC$-frame} if it validates all the theorems of $\MSSIC$.
 
The following theorem states the Kripke completeness of $\MSSIC$ and provides a first-order characterization of $\MSSIC$-frames. 
Using simple syntactic manipulations, it is possible to rewrite all the axioms of $\MSSIC$ into Sahlqvist formulas (see, e.g., \cite[Def.~3.51]{BRV01} for the definition of Sahlqvist formulas in polyadic languages and \cite{GorankoV06} for their generalization to inductive formulas). It then follows from Sahlqvist's theorem (see, e.g., \cite[Thms.~3.54, 4.42]{BRV01}) that all the axioms of $\MSSIC$ are canonical and have a first-order correspondent. Since we also need a characterization of the $\MSSIC$-frames, we instead prefer to show that the SQEMA algorithm \cite{CoGoVa06b} succeeds on computing the first-order correspondents of all the axioms of $\MSSIC$, which guarantees that all the axioms are canonical. We postpone the proof of the theorem to the appendix as it requires lengthy computations.

\begin{theorem}\label{prop:L Kripke complete}
A Kripke frame $(X,T,S)$ is an $\MSSIC$-frame iff for all $x,y,z,w \in X$ we have 
\begin{enumerate}
\item\label{prop:L Kripke complete:item1} $Txxx$;
\item\label{prop:L Kripke complete:item2} $Txyz$ implies $Txzy$;
\item\label{prop:L Kripke complete:item3} the binary relation $E_T$ defined by $x E_T y$ iff $\exists z \, (Txyz)$ is an equivalence relation;
\item\label{prop:L Kripke complete:item4} $Txyz$ and $xE_Tw$ imply $Twyz$;
\item\label{prop:L Kripke complete:item5} $Txyz$ and $ySw$ imply that there is $u \in X$ such that $Txwu$ and $zSu$.
\end{enumerate}
Moreover, $\MSSIC$ is Kripke complete: a formula $\phi$ is a theorem of $\MSSIC$ iff it is valid in all $\MSSIC$-frames.
\end{theorem}

We prove that we can restrict our attention to $\MSSIC$-frames containing a single $E_T$-equivalence class.

\begin{definition}
We call an $\MSSIC$-frame \emph{simple} if it consists of a single $E_T$-equivalence class. 
\end{definition}

\begin{proposition}\label{prop:ET classes generated subframes}
Each $E_T$-equivalence class of an $\MSSIC$-frame is a generated subframe.
\end{proposition}

\begin{proof}
It is sufficient to show that $Txyz$ implies $y,z \in E_T[x]$ and that $xSy$ implies $xE_Ty$ for any $x,y,z \in X$. If $Txyz$, then $xE_Ty$ by definition of $E_T$. By \cref{prop:L Kripke complete}\eqref{prop:L Kripke complete:item2}, if $Txyz$, then $Txzy$, and so $xE_Tz$. Thus, $y,z \in E_T[x]$. Suppose that $xSy$. By \cref{prop:L Kripke complete}\eqref{prop:L Kripke complete:item1}, $Txxx$, and hence there is $u \in X$ such that $Txyu$ and $xSu$ because of \cref{prop:L Kripke complete}\eqref{prop:L Kripke complete:item5}. Then $Txyu$ implies $xE_Ty$ by the definition of $E_T$.
\end{proof}

\begin{corollary}\label{cor: MSSIC complete simple frames}
A formula $\phi$ is a theorem of $\MSSIC$ iff it is valid in all simple $\MSSIC$-frames.
\end{corollary}

\begin{proof}
Each $\MSSIC$-frame is the disjoint union of its $E_T$-equivalence classes, and every $E_T$-class is a generated subframe by \cref{prop:ET classes generated subframes}. Thus, a formula is valid in an $\MSSIC$-frame iff it is valid in all its $E_T$-equivalence classes, which are simple $\MSSIC$-frames. It then follows from the Kripke completeness of $\MSSIC$ (see \cref{prop:L Kripke complete}) that a formula is a theorem of $\MSSIC$ iff it is valid in all simple $\MSSIC$-frames.
\end{proof}

We show that in a simple $\MSSIC$-frame the ternary relation can be replaced by a binary one.

\begin{definition}\label{d:modal contact frame}
A \emph{modal contact frame} is a triple $(X,R,S)$, where $X \neq \emptyset$ and $R,S$ are binary relations such that:
\begin{enumerate}
\item\label{d:modal contact frame:item1} $R$ is reflexive and symmetric,
\item\label{d:modal contact frame:item2} for all $x,y,z \in X$, $xRy$ and $xSz$ imply that there is $w \in X$ such that $zRw$ and $ySw$.
\end{enumerate}
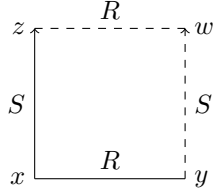
\begin{center}
\begin{tikzpicture}
\node [left] at (0,0) {$x$};
\node [left] at (0,2) {$z$};
\node [right] at (2,0) {$y$};
\node [right] at (2,2) {$w$};
\draw [->] (0,0) -- (0,2);
\draw [dashed, ->] (2,0) -- (2,2);
\draw [dashed] (0,2) -- (2,2);
\draw  (0,0) -- (2,0);
\node [above] at (1,2) {$R$};
\node [above] at (1,0) {$R$};
\node [left] at (0,1) {$S$};
\node [right] at (2,1) {$S$};
\end{tikzpicture}
\end{center}
\end{definition}

\begin{remark}
The second condition in \cref{d:modal contact frame} is known as the \emph{Church-Rosser} (or \emph{confluence}) property (see, e.g., \cite[p.~222]{GKWZ03}). When $R$ is symmetric, as in modal contact frames, the condition is equivalent to the \emph{right commutativity} property: for all $x,y,w \in X$, if $xRy$ and $ySw$, then there exists $z \in X$ such that $zRw$ and $xSz$. 
\end{remark}

If $(X,T,S)$ is a simple $\MSSIC$-frame, then we define a binary relation $R_T$ on $X$ by setting $xR_Ty$ iff $Txxy$. Vice versa, if $(X,R,S)$ is a modal contact frame, let $T_R$ be the ternary relation on $X$ defined by $T_R xyz$ iff $yRz$.

\begin{proposition}\label{prop:corr L-frames and modal con frames}
\mbox{}\begin{enumerate}
\item\label{prop:corr L-frames and modal con frames:item1} If $(X,T,S)$ is a simple $\MSSIC$-frame, then $(X,R_T,S)$ is a modal contact frame and $T=T_{R_T}$.
\item\label{prop:corr L-frames and modal con frames:item2} If $(X,R,S)$ is a modal contact frame, then $(X,T_R,S)$ is a simple $\MSSIC$-frame  and $R=R_{T_R}$.
\item\label{prop:corr L-frames and modal con frames:item3} The mappings $(X,T,S) \mapsto (X,R_T,S)$ and $(X,R,S) \mapsto (X,T_R,S)$ yield a 1-1 correspondence between simple $\MSSIC$-frames and modal contact frames.
\end{enumerate}
\end{proposition}

\begin{proof}
\eqref{prop:corr L-frames and modal con frames:item1}. Let $x,y \in X$. Since all the elements in $X$ are $E_T$-related, \cref{prop:L Kripke complete}\eqref{prop:L Kripke complete:item4} implies that $xR_Ty$ iff there exists $z\in X$ such that $Tzxy$. By \cref{prop:L Kripke complete}\eqref{prop:L Kripke complete:item1}, $R_T$ is reflexive. If $xR_Ty$, then $Txxy$, which implies $Txyx$ by \cref{prop:L Kripke complete}\eqref{prop:L Kripke complete:item2}. So, $xR_Ty$ implies $yR_Tx$, and hence $R_T$ is symmetric. It remains to show that $R_T$ satisfies condition \eqref{d:modal contact frame:item2} of \cref{d:modal contact frame}. Suppose $xR_Ty$ and $xSz$. Then $Txxy$ and \cref{prop:L Kripke complete}\eqref{prop:L Kripke complete:item5} yields $w$ such that $Txzw$ and $ySw$. Thus, $zR_Tw$ and $ySw$. Moreover, $T_{R_T}xyz$ iff $yR_Tz$ iff $Tyyz$ iff $Txyz$, where the last equivalence follows from \cref{prop:L Kripke complete}\eqref{prop:L Kripke complete:item4}.
 
\eqref{prop:corr L-frames and modal con frames:item2}. We first show that all the elements of $X$ are $E_{T_R}$-related. Since $R$ is reflexive, we have $yRy$ for all $y \in X$. Thus, $T_Rxyy$, and so $xE_{T_R}y$ for all $x \in X$. It then follows that $E_{T_R}$ is an equivalence relation. Moreover, $yR_{T_R}z$ iff $T_Ryyz$ iff $yRz$. It is left to show that conditions \eqref{prop:L Kripke complete:item1}--\eqref{prop:L Kripke complete:item5} of \cref{prop:L Kripke complete} hold in $(X, T_R, S)$. That $R$ is reflexive also implies that $T_Rxxx$, and hence \eqref{prop:L Kripke complete:item1} holds. To show \eqref{prop:L Kripke complete:item2}, suppose that $T_Rxyz$. Then $yRz$, and so $zRy$ by the symmetry of $R$. Thus, $T_Rxzy$. To prove \eqref{prop:L Kripke complete:item4}, observe that if $T_Rxyz$, then $yRz$, which implies $T_Rwyz$. To show \eqref{prop:L Kripke complete:item5}, suppose $T_Rxyz$ and $ySw$. Then $yRz$ and $ySw$, which imply that there is $u \in X$ such that $wRu$ and $zSu$. Therefore, $T_Rxwu$ and $zSu$.

\eqref{prop:corr L-frames and modal con frames:item3} is an immediate consequence of \eqref{prop:corr L-frames and modal con frames:item1} and \eqref{prop:corr L-frames and modal con frames:item2}. 
\end{proof}

The 1-1 correspondence of \cref{prop:corr L-frames and modal con frames}\eqref{prop:corr L-frames and modal con frames:item3} allows us to interpret formulas in the language of $\MSSIC$ in modal contact frames. Let $(X,R,S)$ be a modal contact frame. The Boolean connectives and $\Diamond$ are interpreted as in $\MSSIC$-frames, and for any valuation $v$ we have
\begin{align*}
x \vDash_v \phi \rightsquigarrow \psi \quad 
& \text{iff} \quad \forall y,z \in X \, (T_Rxyz \text{ and } y \vDash_v \phi \text{ imply } z \vDash_v \psi )\\
& \text{iff} \quad \forall y,z \in X \, (yRz \text{ and } y \vDash_v \phi \text{ imply } z \vDash_v \psi ).
\end{align*}
Since $x$ does not play any role in the condition on the right-hand side, we obtain that 
\[
x \vDash_v \phi \rightsquigarrow \psi \quad \text{iff} \quad R[v(\phi)] \subseteq v(\psi),
\]
which implies that
\[
v(\phi \rightsquigarrow \psi) = \begin{cases}
  X  & \text{ if } R[v(\phi)] \subseteq v(\psi), \\
  \emptyset & \text{ otherwise}.
\end{cases}
\]

The following corollary, stating the Kripke completeness of $\MSSIC$ with respect to modal contact frames, is an immediate consequence of \cref{cor: MSSIC complete simple frames} and \cref{prop:corr L-frames and modal con frames}.

\begin{corollary}\label{c:Kripke compl modal cont frames}
A formula $\phi$ is a theorem of $\MSSIC$ iff it is valid in all modal contact frames.
\end{corollary}

We now turn to morphisms between frames with the goal of describing the maps between modal contact frames that correspond to p-morphisms between $\MSSIC$-frames. Recall (see, e.g., \cite[Def.~2.12]{BRV01}) that a map $f \colon (X,T,S) \to (X',T',S')$ between two $\MSSIC$-frames is a \emph{p-morphism} if it satisfies the following conditions:
\begin{enumerate}[leftmargin=1.2cm]
\item[(MT1)] for all $x,y,z \in X$, if $Txyz$, then $T'f(x) f(y) f(z)$;
\item[(MT2)] for all $x \in X$ and $y',z' \in X'$, if $T' f(x) y' z'$, then there are $y,z \in X$ such that $Txyz$, $f(y)=y'$, and $f(z)=z'$;
\item[(MS1)] for all $x,y \in X$, if $xSy$, then $f(x) S' f(y)$;
\item[(MS2)] for all $x \in X$ and $y'\in X'$, if $f(x) S' y'$, then there is $y \in X$ such that $xSy$ and $f(y)=y'$.
\end{enumerate}

Our next goal is to describe the corresponding morphisms of modal contact frames.

\begin{definition}
We call a map $f \colon (X,R,S) \to (X',R',S')$ between modal contact frames a \emph{regular stable p-morphism} if it satisfies conditions (MS1), (MS2), and
\begin{enumerate}[leftmargin=1.2cm]
\item[(MR1)] for all $x,y \in X$, if $xRy$, then $f(x) R' f(y)$;
\item[(MR2)] for all $x',y' \in X'$, if $x'R'y'$, then there are $x,y \in X$ such that $xRy$, ${f(x)=x'}$, and $f(y)=y'$.
\end{enumerate}
\end{definition}

\begin{proposition}\label{prop:correspondence morphisms}
Let $(X,R,S)$ and $(X',R',S')$ be modal contact frames and $f \colon X \to X'$ a map. 
Then $f \colon (X,R,S) \to (X',R',S')$ is a regular stable p-morphism iff $f \colon (X,T_R,S) \to (X',T_{R'},S')$ is a p-morphism.
\end{proposition}

\begin{proof}
Suppose that $f \colon (X,R,S) \to (X',R',S')$ is a regular stable p-morphism. We need to show that $f$ satisfies the conditions (MT1) and (MT2) with respect to $T_R$ and $T_{R'}$.  To show that (MT1) holds, let $x,y,z \in X$ with $T_Rxyz$. Then $yRz$, and so $f(y)R'f(z)$ by (MR1). Thus, $T_{R'}f(x)f(y)f(z)$. We show (MT2) holds. Let $x \in X$ and $y',z' \in X'$ be such that $T_{R'}f(x)y'z'$. Then $y'R'z'$, and by (MR2) there are $y,z \in X$ such that $f(y)=y'$, $f(z)=z'$, and $yRz$, which implies $T_Rxyz$.

Assume now that $f \colon (X,T_R,S) \to (X',T_{R'},S')$ is a p-morphism. We show that $f$ satisfies the conditions (MR1) and (MR2). To prove (MR1), let $x,y \in X$ such that $xRy$. Then $T_Rxxy$, which implies $T_{R'}f(x)f(x)f(y)$ by (MT1). It follows that $f(x)R'f(y)$. To show (MR2), let $x',y' \in X'$. Since $X \neq \emptyset$, there is $w \in X$, and so $T_{R'}f(w)x'y'$. By (MT2), there exist $x,y \in X$ such that $f(x)=x'$, $f(y)=y'$, and $T_Rwxy$, which implies $xRy$.
\end{proof}

Since regular stable p-morphism between modal contact frames correspond to p-morphisms between $\MSSIC$-frames, they preserve and reflect truth of formulas. More precisely, if $f \colon (X,R,S) \to (X',R',S')$ is a regular stable p-morphism and $v$ a valuation on $X'$, then $w=f^{-1} \circ v$ is a valuation on $X$ such that
\begin{align*}
(X,R,S), \, x \vDash_w \phi  \quad & \text{iff} \quad (X',R',S'), \, f(x) \vDash_v \phi
\end{align*}
for any formula $\phi$.

To prove the admissibility of $\Pi_2$-rules in $\MSSIC$ we will employ the following lemma, which is an immediate consequence of \cref{c:Kripke compl modal cont frames}.

\begin{lemma}\label{l:admissibility countermodels}
Let $\rho$ be the $\Pi_2$-rule 
\[
\inference{F(\overline{\phi}/\overline{x},\overline{p}) \to \chi}{G(\overline{\phi}/\overline{x}) \to \chi},
\]
where $F(\overline{x},\overline{p}),G(\overline{x})$ are formulas in the language of $\MSSIC$. The rule $\rho$ is admissible in $\MSSIC$ iff for any $\overline{\phi},\chi$ formulas, modal contact frame $(X,R,S)$, and valuation $v$ on $X$ such that 
\[
(X,R,S) \nvDash_v G(\overline{\phi}/\overline{x}) \to \chi,
\] 
there exists a modal contact frame $(X',R',S')$ and a valuation $w$ on $X'$ such that
\[
(X',R',S') \nvDash_w F(\overline{\phi}/\overline{x},\overline{p}) \to \chi.
\]
\end{lemma}

The following theorems show that all the $\Pi_2$-rules of $\UMSSIC$ (see \cref{def:MSSICU}) are admissible in $\MSSIC$.

\begin{theorem}\label{t:rho6 admissible}
The $\Pi_2$-rule
\begin{equation*}
\inference{(\phi\rightsquigarrow p)\land(p\rightsquigarrow\psi)\to\chi}{(\phi\rightsquigarrow\psi)\to\chi}\tag{$\rho$7}
\end{equation*}
is admissible in $\MSSIC$.
\end{theorem}

\begin{proof}
Let $(X,R,S)$ be a modal contact frame and $v$ a valuation on $X$ such that $(X,R,S) \nvDash_v (\phi\rightsquigarrow\psi)\to\chi$. We define a new modal contact frame $(X',R',S')$. Set $X'= \{ (x_1,x_2) \in X \times X : x_1Rx_2 \}$. The binary relation $R'$ is given by $(x_1,x_2)R'(y_1,y_2)$ iff $\{ x_1,x_2 \} = \{ y_1,y_2 \}$. Define $S'$ by $(x_1,x_2)S'(y_1,y_2)$ iff $x_1Sy_1$ and $x_2Sy_2$. We show that $(X',R',S')$ is a modal contact frame. It is an immediate consequence of the definition of $R'$ that $R'$ is reflexive and symmetric. It remains to show that condition \eqref{d:modal contact frame:item2} of \cref{d:modal contact frame} holds. Suppose that $(x_1,x_2)R'(y_1,y_2)$ and $(x_1,x_2)S'(z_1,z_2)$. If $(x_1,x_2)=(y_1,y_2)$, then there is nothing to prove. If $(x_1,x_2) \neq (y_1,y_2)$, then $x_1 \neq x_2$ and the definition of $R'$ implies that $y_1=x_2$ and $y_2=x_1$. Thus, we have $(z_1,z_2)R'(z_2,z_1)$ and $(y_1,y_2)=(x_2,x_1) S' (z_2,z_1)$, where the last relation holds because $(x_1,x_2)S'(z_1,z_2)$. Consequently, $(X',R',S')$ is a modal contact frame.

Let $f \colon X' \to X$ be defined by $f(x_1,x_2)=x_1$. We show that $f$ is a regular stable p-morphism. 
If $(x_1,x_2)S'(y_1,y_2)$, then it follows from the definition of $S'$ that $f(x_1,x_2)=x_1Sy_1=f(y_1,y_2)$. Suppose that $f(x_1,x_2)Sy_1$, then $x_1Rx_2$ and $x_1Sy_1$. Since $(X,R,S)$ is a modal contact frame, we have that there exists $y_2 \in X$ such that $y_1Ry_2$ and $x_2Sy_2$. Thus, $(x_1,x_2)S'(y_1,y_2)$ and $f(y_1,y_2)=y_1$. Therefore, $f$ satisfies (MS1) and (MS2).
If $(x_1,x_2)R'(y_1,y_2)$, then $x_1Rx_2$ and $y_1 \in \{x_1,x_2\}$, and hence $f(x_1,x_2)=x_1Ry_1=f(y_1,y_2)$. If $x,y \in X$ with $xRy$, then $(x,y)R'(y,x)$, $f(x,y)=x$, and $f(y,x)=y$. Thus, $f$ satisfies (MR1) and (MR2).

Let $w$ be the valuation on $X'$ given by $w(q)=f^{-1}[v(q)]$ for any propositional variable $q$ distinct from $p$ and $w(p) = \{(x_1,x_2) : x_1 \vDash_v \phi \text{ or } x_2 \vDash_v \phi \}$. It follows from the definitions of $R'$ and $w$ that $w(p)=R'[f^{-1}[v(\phi)]]=R'[w(\phi)]$.
Since $(X,R,S) \nvDash_v (\phi\rightsquigarrow\psi)\to\chi$, there exists $a \in X$ such that $a \vDash_v \phi\rightsquigarrow\psi$ but $a \nvDash_v \chi$. That $a \vDash_v \phi\rightsquigarrow\psi$ simply means $R[v(\phi)] \subseteq v(\psi)$. Let $a'=(a,a)$. We have that $a' \in X'$ because $R$ is reflexive. We show that $a' \nvDash_w (\phi\rightsquigarrow p)\land(p\rightsquigarrow\psi)\to\chi$. This requires to show that $a' \vDash_w \phi\rightsquigarrow p$, $a' \vDash_w p\rightsquigarrow\psi$, and $a' \nvDash \chi$. Thus, we need to prove that $R'[w(\phi)] \subseteq w(p)$, $R'[w(p)] \subseteq w(\psi)$, and $a' \nvDash \chi$. Since, as we observed above, $w(p)=R'[w(\phi)]$, the first inclusion holds. To prove the second inclusion, it is sufficient to show that $w(p) \subseteq w(\psi)$ because the definitions of $R'$ and $w$ imply that $R'[w(p)]= w(p)$. Assume that $(x_1,x_2) \in w(p)$. If $x_1 \in v(\phi)$, then $x_1 \in R[v(\phi)] \subseteq v(\psi)$. Otherwise, $x_2 \in v(\phi)$ and $x_1 \in R[v(\phi)] \subseteq v(\psi)$ because $x_1Rx_2$ and $R$ is symmetric. In either case, $(x_1,x_2) \in f^{-1}[v(\psi)]=w(\psi)$. This proves that $w(p) \subseteq w(\psi)$. Since $f(a')=a \nvDash_v \chi$, we have $a' \nvDash_w \chi$. Consequently, $a' \nvDash_w (\phi\rightsquigarrow p)\land(p\rightsquigarrow\psi)\to\chi$, which implies $(X',R',S') \nvDash_w (\phi\rightsquigarrow p)\land(p\rightsquigarrow\psi)\to\chi$. Therefore, ($\rho$7) is admissible in $\MSSIC$ by \cref{l:admissibility countermodels}.
\end{proof} 

\begin{theorem}\label{t:rho7 admissible}
The $\Pi_2$-rule
\begin{equation*}
\inference{p \land (p\rightsquigarrow\phi)\to\chi}{\phi\to\chi}\tag{$\rho$8}
\end{equation*}
is admissible in $\MSSIC$.
\end{theorem}

\begin{proof}
Let $(X,R,S)$ be a modal contact frame and $v$ a valuation on $X$ such that $(X,R,S) \nvDash_v \phi\to\chi$. We define a new modal contact frame $(X',R',S')$. Set $X'=\{1,2\} \times X$. The binary relation $R'$ is given by $(i,x) R' (j,y)$ iff either $i=j=1$ and $x=y$ or $i=j=2$ and $xRy$. The binary relation $S'$ is defined by $(i,x) S' (j,y)$ iff $i=j$ and $xSy$.
Thus, $X'$ is obtained by taking the disjoint union of two copies of $(X,R,S)$ and replacing $R$ with the identity relation in the first copy.
We show $(X',R',S')$ is a modal contact frame. It is immediate to see that $R'$ is reflexive and symmetric. 
It remains to show that condition \eqref{d:modal contact frame:item2} of \cref{d:modal contact frame} holds.
Let $(i,x), (j,y), (h,z) \in X'$ such that $(i,x) R' (j,y)$ and $(i,x) S' (h,z)$. By the definitions of $R'$ and $S'$ we have $i=j=h$. If $i=j=h=1$, then $x=y$ and $xSz$. So, in this case $(1,z) R' (1,z)$ and $(1,x) S' (1,z)$. Otherwise, if $i=j=h=2$, then $xRy$ and $xSz$. Since $(X,R,S)$ is a modal contact frame, there exists $u \in X$ such that $zRu$ and $ySu$. Then, $(2,z) R' (2,u)$ and $(2,y) S' (2,u)$. This proves that $(X',R',S')$ is a modal contact frame.

Let $f \colon X' \to X$ be defined by $f(i,x)=x$. We show that $f$ is a regular stable p-morphism. If $(i,x) S' (j,y)$, then $xSy$, so $f(i,x) S f(j,y)$. If $x=f(i,x) S y$, then $(i,x) S (i,y)$ and $f(i,y)=y$. Thus, $f$ satisfies (MS1) and (MS2). If $(i,x) R' (j,y)$, then either $x=y$ or $xRy$. In both cases we have $xRy$, and so $f(i,x) R f(j,y)$. Finally, suppose $xRy$. Then $(2,x) R' (2,y)$, $f(2,x)=x$, and $f(2,y)=y$. Thus, $f$ satisfies (MR1) and (MR2).

Let $w$ be the valuation on $X'$ given by $w(q)=f^{-1}[v(q)]$ for any propositional variable $q$ distinct from $p$ and $w(p) = \{(1,x) : x \vDash_v \phi\}$.
Since $(X,R,S) \nvDash_v \phi\to\chi$, there exists $a \in X$ such that $a \vDash_v \phi$ but $a \nvDash_v \chi$. 
Let $a' = (1,a) \in X'$.
We prove that $a' \nvDash_w p \land (p\rightsquigarrow\phi)\to\chi$. This requires to show that $a' \vDash_w p$, $a' \vDash_w p\rightsquigarrow\phi$, and $a' \nvDash_w \chi$. Since $a \vDash_v \phi$ and $a'=(1,a)$, the definition of $w(p)$ yields that  $a' \vDash_w p$. We have that $f^{-1}[v(\phi)] = w(\phi)$ because $f$ is a regular stable p-morphism and $\phi$ does not contain $p$. Therefore,
\begin{align*}
R'[w(p)] &= w(p) = \{(1,x) : x \vDash_v \phi\}\\
&\subseteq \{(i,x) : x \vDash_v \phi \text{ and } i\in \{1,2\}\} = f^{-1}[v(\phi)] = w(\phi),
\end{align*}
which implies that $a' \in X' = w(p \rightsquigarrow \phi)$. Finally, $a' \nvDash_w \chi$ because $w(\chi)=f^{-1}[v(\chi)]$ and $f(a')=a \nvDash_v \chi$.
Consequently, $a' \nvDash_w p \land (p\rightsquigarrow\phi)\to\chi$, which implies $(X',R',S') \nvDash_w p \land (p\rightsquigarrow\phi)\to\chi$. Therefore, ($\rho$8) is admissible in $\MSSIC$ by \cref{l:admissibility countermodels}.
\end{proof}

\begin{theorem}\label{t:uc admissible}
The $\Pi_2$-rule
\begin{equation*}
\inference{\Box p \land (p\rightsquigarrow\phi)\to\chi}{\Box\phi\to\chi}\tag{UC}
\end{equation*}
is admissible in $\MSSIC$.
\end{theorem}

\begin{proof}
Let $(X,R,S)$ be a modal contact frame and $v$ a valuation on $X$ such that $(X,R,S) \nvDash_v \Box \phi \to \chi$.
Let $(X',R',S')$, $f \colon X' \to X$, and $w$ be defined as in the proof of \cref{t:rho7 admissible}.

Since $(X,R,S) \nvDash_v \Box \phi \to \chi$, there exists $a \in X$ such that $a \vDash_v \Box \phi$ but $a \nvDash_v \chi$. 
Let $a' = (1,a) \in X'$. We prove that  $a' \nvDash_w \Box p \land (p\rightsquigarrow\phi) \to \chi$. This requires to show that $a' \vDash_w \Box p$, $a' \vDash_w p\rightsquigarrow\phi$, and $a' \nvDash_w \chi$.
The proofs that $a' \vDash_w p\rightsquigarrow\phi$ and $a' \nvDash_w \chi$ are the same as in the proof of \cref{t:rho7 admissible}.
Since $a \vDash_v \Box \phi$, we have 
\[
S'[a']=\{ (1,x) : aSx \} \subseteq \{ (1,x) : x \vDash_v \phi \} = w(p),
\]
and hence $a'=(1,a) \vDash_w \Box p$. Consequently, $a' \nvDash_w \Box p \land (p\rightsquigarrow\phi) \to \chi$, which implies $(X',R',S') \nvDash_w \Box p \land (p\rightsquigarrow\phi) \to \chi$. Therefore, (UC) is admissible in $\MSSIC$ by \cref{l:admissibility countermodels}.
\end{proof}

As an immediate consequence of the definition of $\UMSSIC$ and \cref{t:rho6 admissible,t:rho7 admissible,t:uc admissible}, we obtain:

\begin{theorem}\label{thm:MSSICU=MSSIC}
$\UMSSIC$ coincides with $\MSSIC$.
\end{theorem}

Let $\AMDeV$\label{AMDeV} denote the class of finitely additive modal de Vries algebras. We have the following completeness results.

\begin{corollary}\label{cor:MSSIC strongly complete MDVA}
For a set of formulas $\Gamma$ and a formula $\phi$, we have
\[
\Gamma\vdash_{\MSSIC}\phi \iff \Gamma\vDash_{\AMDeV}\phi \iff \Gamma\vDash_{\MKHaus}\phi.
\]
\end{corollary}

\begin{proof}
We first show that $\Gamma\vdash_{\MSSIC}\phi$ iff $\Gamma\vDash_{\AMDeV}\phi$. Since all the axioms of $\MSSIC$ are valid in any finitely additive modal de Vries algebra, the left-to-right implication is a straightforward verification. The right-to-left implication holds because $\UMSSIC$ is strongly complete with respect to $\UMDeV$ (see \cref{thm:completeness upper modal de vries}), which is a subclass of $\AMDeV$. As an immediate consequence of \cref{cor:MSSICU completeness MKHaus,thm:MSSICU=MSSIC}, we have that $\Gamma\vdash_{\MSSIC}\phi$ iff $\Gamma\vDash_{\MKHaus}\phi$.
\end{proof}

Let ($\rho$9) be the $\Pi_2$-rule
\begin{equation*}
\inference{(p\rightsquigarrow p) \land (\phi\rightsquigarrow p)\land (p\rightsquigarrow \psi)\to\chi}{(\phi\rightsquigarrow\psi)\to\chi}.
\end{equation*}
It is shown in \cite[Lemma~6.14]{BeBeSaVe19} that $\Pi(\rho 9)$ holds in a contact algebra $(B, \prec)$ iff the following condition holds
\begin{itemize}
\item[(S9)] $a \prec b$ implies there is $c$ with $c\prec c$ and $a \prec c \prec b$.
\end{itemize}

It is proved in \cite[Lemma~4.11]{Be10} that a de Vries algebra satisfies (S9) iff its dual space is zero-dimensional, and hence a Stone space. For this reason, the de Vries algebras satisfying (S9) are called \emph{zero-dimensional} in \cite{Be10}. We will also call any contact algebra satisfying (S9) zero-dimensional.

\begin{theorem}\label{thm:rho9 admissible}
The $\Pi_2$-rule \emph{($\rho$9)}
is admissible in $\MSSIC$.
\end{theorem}

\begin{proof}
Let $(X,R,S)$ be a modal contact frame and $v$ a valuation on $X$ such that $(X,R,S) \nvDash_v (\phi\rightsquigarrow\psi)\to\chi$.
Let $(X',R',S')$, $f \colon X' \to X$, and the valuation $w$ be defined as in the proof of \cref{t:rho6 admissible}. It is shown in the proof of \cref{t:rho6 admissible} that  $R'[w(p)]=w(p)$ and that there is an element $a'$ such that $a' \vDash_w \phi\rightsquigarrow p$, $a' \vDash_w p\rightsquigarrow\psi$, and $a' \nvDash \chi$. From $R'[w(p)]=w(p)$ it follows that $a' \vDash_w p\rightsquigarrow p$. Thus, $a' \nvDash_w (p\rightsquigarrow p) \land (\phi\rightsquigarrow p)\land (p\rightsquigarrow \psi)\to\chi$, and hence  
\[
(X',R',S') \nvDash_w (p\rightsquigarrow p) \land (\phi\rightsquigarrow p)\land (p\rightsquigarrow \psi)\to\chi. 
\]
Therefore, ($\rho$9) is admissible in $\MSSIC$ by \cref{l:admissibility countermodels}.
\end{proof}

Let $\ZUMComp$\label{ZUMComp} be the class of upper continuous modal compingent algebras that are zero-dimensional. The following theorem, which states the completeness of $\MSSIC$ with respect to $\ZUMComp$, is an immediate consequence of \cref{thm:completeness pi2 rules,thm:MSSICU=MSSIC,thm:rho9 admissible}.

\begin{theorem}\label{thm:MSSIC complete wrt zMcompa}
For a formula $\phi$, we have
\[
\vdash_{\MSSIC}\phi \iff \ \vDash_{\ZUMComp}\phi.
\]
\end{theorem}

We now show that (S9) is preserved by MacNeille completions. We will obtain as a consequence that $\MSSIC$ is also complete with respect to  zero-dimensional modal de Vries algebras, and hence with respect to descriptive frames.

\begin{proposition}\label{prop:zerodim MacNeille}
Let $(B, \prec)$ be a compingent algebra satisfying \emph{(S9)}. Then ${(\overline{B}, \prec)}$ is a zero-dimensional de Vries algebra.
\end{proposition}

\begin{proof}
Suppose that $\alpha \prec \beta$. Then by definition of $\prec$ on the MacNeille completion, there are $a,b \in B$ such that $\alpha \le a \prec b \le \beta$. By (S9), there is $c \in B$ such that $c \prec c$ and $a \prec c \prec b$. Since the inclusion of $B$ into $\overline{B}$ preserves $\prec$, we have $\alpha \prec c \prec \beta$ and $c \prec c$ in $\overline{B}$. Therefore, $(\overline{B}, \prec)$ satisfies (S9).
\end{proof}

\begin{definition}
A modal compact Hausdorff space $(X,R)$ is called a \emph{descriptive frame} or a \emph{modal space} if $X$ is a Stone space. 
\end{definition}
 
The dual equivalence between the category of upper continuous modal de Vries algebras and the category of modal compact Hausdorff spaces restricts to a dual equivalence between the categories of zero-dimensional upper continuous modal de Vries algebras and the category of descriptive frames (see \cite[Thm.~6.3]{BeBeHa15}).
Let $\ZAMDeV$\label{ZAMDeV} be the class of zero-dimensional finitely additive modal de Vries algebras and  $\ZUMDeV$\label{ZUMDeV} its subclass of zero-dimensional upper continuous modal de Vries algebras. Let also $\DFrm$ be the class of modal de Vries algebras of the form $(\RO(X), \prec, \Diamond^U)$ for some descriptive frame $(X,R)$.
The following theorem establishes the strong completeness of $\MSSIC$ with respect to $\ZUMDeV$, $\ZAMDeV$, and $\DFrm$.

\begin{theorem}
For a set of formulas $\Gamma$ and a formula $\phi$, we have
\[
\Gamma \vdash_{\MSSIC} \phi \iff \Gamma \vDash_{\ZUMDeV} \phi \iff \Gamma \vDash_{\ZAMDeV} \phi \iff \Gamma \vDash_{\DFrm} \phi.
\]
\end{theorem}

\begin{proof}
By \cref{thm:MSSIC complete wrt zMcompa}, $\MSSIC$ is sound and complete with respect to $\ZUMComp$. That $\MSSIC$ is strongly sound and complete with respect to $\ZUMComp$ follows by arguing as in \cref{thm:completeness upper modal de vries}. Then, by using \cref{prop:zerodim MacNeille,Lemma:MN:Completion} and arguing as in \cref{thm:completeness upper modal de vries}, we obtain that $\MSSIC$ is strongly sound and complete with respect to $\ZUMDeV$. By arguing as in \cref{cor:MSSIC strongly complete MDVA}, it follows that $\MSSIC$ is strongly sound and complete also with respect to $\ZAMDeV$. That $\MSSIC$ is strongly sound and complete with respect to $\DFrm$ follows from the dual equivalence between the categories of  zero-dimensional upper continuous modal de Vries algebras and the category of descriptive frames, which yields that $\DFrm \subseteq \ZUMDeV$ and each member of $\DFrm$ is isomorphic to one of $\ZUMDeV$.
\end{proof}

By \cite[Thm.~5.8]{BeBeHa15}, to each modal compact Hausdorff space $(X,R)$ it is possible to associate the modal de Vries algebra $(\RO(X), \prec, \Diamond^L)$, where
$\Diamond^L O = \int (\cl (R^{-1}[O]))$.
The operator $\Diamond^L$ is de Vries additive and lower continuous, where we recall from \cref{def:de vries mult and lower continuous} that $\Diamond$ on $(B, \prec)$ is lower continuous if $\Diamond a = \bigvee \{\Diamond b: b\prec a\}$ for every $a \in B$. 
Lower continuous modal de Vries algebras form a category dually equivalent to the category of modal compact Hausdorff spaces \cite[Thm.~5.14]{BeBeHa15} and equivalent to the categories of modal de Vries algebras and of upper continuous modal de Vries algebras \cite[Thm.~4.23]{BeBeHa15}.
By \cref{prop:box and diamond}, $\Diamond^L$ corresponds to an upper continuous de Vries multiplicative operator $\Box^U$ on $\RO(X)$, which is given by  $\Box^U O = \int (\Box_R(\cl(O)))$, where $\Box_R Y=X \setminus R^{-1}[X \setminus Y]$ for $Y \subseteq X$.

\begin{remark}\label{rem:lower continous calculus}
As for future work, we mention that
it should  also be possible to obtain a calculus strongly sound and complete with respect to the class of lower continuous modal de Vries algebras. However, unlike upper continuity, the lower continuity of a de Vries additive operator does not imply finite additivity (see \cite[Ex.~4.16(1)]{BeBeHa15}). So, the operator $\Diamond$ is not necessarily finitely additive in a lower continuous de Vries algebra. Thus, Axiom (K) is not valid in the class of lower continuous modal de Vries algebras and must be replaced by several axioms and inference rules. As a result, the definition of such a calculus would be more involved than the one of $\UMSSIC$. For this reason, we leave this investigation to  future work.
\end{remark}

We end the section by showing that the $\Pi_2$-rule corresponding to lower continuity is admissible in $\MSSIC$. Let (LC) be the $\Pi_2$-rule

\begin{equation*}
\inference{\Diamond p \land (p\rightsquigarrow\phi)\to\chi}{\Diamond\phi\to\chi}.
\end{equation*}
A proof similar to the one of \cref{prop:Pi2-correspondence}\eqref{prop:Pi2-correspondence:item2}
yields that if $(B, \prec, \Diamond)$ is a finitely additive modal contact algebra and $(B, \rightsquigarrow, \Diamond)$ the corresponding simple $\MSSIC$-algebra, then $(B, \rightsquigarrow, \Diamond)$ satisfies $\Pi\text{(LC)}$ iff $(B, \prec, \Diamond)$ is lower continuous.

\begin{theorem}\label{thm:LC admissible}
The $\Pi_2$-rule \emph{(LC)} is admissible in $\MSSIC$.
\end{theorem}

\begin{proof}
Let $(X,R,S)$ be a modal contact frame and $v$ a valuation on $X$ such that $(X,R,S) \nvDash_v \Diamond \phi \to \chi$. 
Then there is $a \in X$ such that $a \vDash_v \Diamond \phi$ but $a \nvDash_v \chi$. 
Let $(X',R',S')$, $f \colon X' \to X$, and $w$ be defined as in the proof of \cref{t:rho7 admissible}. 

Let $a' = (1,a) \in X'$. We show that $a' \nvDash_w \Diamond p \land (p\rightsquigarrow\phi)\to\chi$. 
The proofs that $a' \vDash_w p \rightsquigarrow \phi$ and $a' \nvDash_w \chi$ are the same as in the proof of \cref{t:rho7 admissible}. It remains to show that $a' \vDash_w \Diamond p$. Since $a \vDash_v \Diamond \phi$, there exists $b \in X$ such that $aSb$ and $b \vDash_v \phi$. If $b' = (1,b)$, then $b' \in \{(1,x) : x \vDash_v \phi\} = w(p)$. Thus, $a'S'b'$ and $b' \vDash_w p$, which imply that $a' \vDash_w \Diamond p$. Consequently, $a' \nvDash_w \Diamond p \land (p\rightsquigarrow\phi)\to\chi$, which yields $(X',R',S') \nvDash_w \Diamond p \land (p\rightsquigarrow\phi)\to\chi$. Therefore, (LC) is admissible in $\MSSIC$ by \cref{l:admissibility countermodels}.
\end{proof}

It is possible to prove analogues of \cref{thm:sound and complete,thm:strong sound and complete UMC} for lower continuous finitely additive modal compingent algebras. It then follows from \cref{thm:LC admissible} that $\MSSIC$ is strongly sound and complete with respect to lower continuous finitely additive modal compingent algebras. However, it is not clear if $\MSSIC$ is strongly sound and complete with respect to lower continuous modal de Vries algebras because that would require to prove an analogue of \cref{Lemma:MN:Completion} for lower continuous modal compingent algebras.

The table below lists all the classes of contact algebras that we considered and the page in which they were introduced.

\begin{center}\label{table}
\begin{tabular}{|p{.6in}p{3.25in}p{.55in}|}\hline
\textbf{Class} & \textbf{Elements} & \textbf{Location}\\ \hline
$\Con$ & contact algebras & page~\pageref{Con,Comp,DeV} \\
$\Comp$ &  compingent algebras & page~\pageref{Con,Comp,DeV} \\
$\DeV$ &  de Vries algebras & page~\pageref{Con,Comp,DeV} \\
$\AMCon$ &  finitely additive modal contact algebras & page~\pageref{AMCon} \\
$\UMComp$ & upper continuous modal compingent algebras & page~\pageref{UMComp} \\
$\UMDeV$ & upper continuous modal de Vries algebras & page~\pageref{UMDeV} \\
$\AMDeV$ & finitely additive modal de Vries algebras & page~\pageref{AMDeV} \\
$\ZUMComp$ & zero-dimensional upper continuous modal compingent algebras & page~\pageref{ZUMComp} \\
$\ZAMDeV$ & zero-dimensional finitely additive modal de Vries algebras & page~\pageref{ZAMDeV} \\
$\ZUMDeV$ & zero-dimensional upper continuous modal de Vries algebras & page~\pageref{ZUMDeV} \\
\hline
\end{tabular}
\end{center}

\appendix

\section{Kripke completeness of $\MSSIC$}\label{appendix}

In this appendix we present a proof of \cref{prop:L Kripke complete}, which states the Kripke completeness of $\MSSIC$ and provides a first-order characterization of the $\MSSIC$-frames. We utilize the SQEMA algorithm \cite{CoGoVa06} and its extension to polyadic formulas \cite{CoGoVa06b}. We show that the algorithm succeeds on all the axioms of $\MSSIC$ and compute their locally first-order correspondent formulas. The success of the algorithm guarantees that all the axioms are canonical, and that $\MSSIC$ is Kripke complete.

In order to execute the algorithm we rewrite the axioms of $\MSSIC$ in the modal propositional language $\mathcal{L}_{\nabla \Box}$ containing two unary modalities $\univ$ and $\Box$, and a binary modality $\nabla$. The binary modality $\rightsquigarrow$ is replaced by $\nabla$ by setting $\nabla(\phi,\psi)=\neg \phi \rightsquigarrow \psi$ and $\univ \phi$, which is defined as an abbreviation of $\top \rightsquigarrow \phi$ in $\MSSIC$, replaces $\nabla(\bot, \phi)$. 
We use the abbreviations $\langle \exists \rangle \varphi, \Diamond \varphi, \Delta(\varphi,\psi)$ for $\neg \univ \neg \varphi, \neg \Box \neg \varphi, \neg \nabla(\neg \varphi, \neg\psi)$, respectively.
We add Axiom (A0) that defines $\univ$ in terms of $\nabla$. It is straightforward to check that the set of corresponding axioms of $\MSSIC$ is the following.
\begin{itemize}
\item[(A0)] $\univ p \biimpl \nabla(\bot,p)$;
\item[(A1)] $\nabla(\top,p) \wedge \nabla(p,\top)$;
\item[(A2)] $\nabla(p \wedge q,r) \biimpl \nabla(p,r) \wedge \nabla(q,r)$;
\item[(A3)] $\nabla(p,q \wedge r) \biimpl \nabla(p,q) \wedge \nabla(p,r)$;
\item[(A4)] $\nabla(p,q) \impl (p \vee q)$;
\item[(A5)] $\nabla(p,q) \biimpl \nabla(q,p)$;
\item[(A8)] $\univ p \impl \univ \univ p$;
\item[(A9)] $\ne\univ p \impl \univ\neg\univ p$;
\item[(A10)] $\nabla(p,q) \biimpl \univ \nabla(p,q)$;
\item[(A11)] $\univ p \impl \nabla(\univ p,\bot)$;
\item[(K)] $\Box(p \to q) \to (\Box p \to \Box q)$;
\item[(Add)] $\nabla(p,q) \impl \nabla(\Diamond p, \Box q)$.
\end{itemize}

Formulas of $\mathcal{L}_{\nabla \Box}$ will be interpreted in Kripke frames of the form $(X,E,T,S)$, where $E$ and $S$ are binary relations, and $T$ is a ternary relation. The modality $\Box$ is interpreted as in \cref{Sec:Admissibility}, while the interpretations of $\univ$ and $\nabla$ are consequences of their definitions in terms of $\rightsquigarrow$: if $x \in X$ and $v$ is a valuation on $X$, we define 
\begin{align*}
x \vDash_v \univ \phi \quad & \text{iff} \quad \forall y \in X \, (xEy \text{ implies } y \vDash_v \phi )\\
x \vDash_v \nabla(\phi,\psi) \quad & \text{iff} \quad \forall y, z \in X \, (Txyz \text{ implies } y \vDash_v \phi \text{ or } z \vDash_v \psi)\\
x \vDash_v \Box \phi \quad & \text{iff} \quad \forall y \in X \, (xSy \text{ implies } y \vDash_v \phi).
\end{align*} 
The \emph{reversive extension} of $\mathcal{L}_{\nabla \Box}$ is obtained by extending $\mathcal{L}_{\nabla \Box}$ with the unary modalities $\univin$ and $\Boxin$, and the binary modalities $\nablain{1}$, $\nablain{2}$. We will also use the abbreviations (for $i=1,2$)
\begin{align*}
\exin \phi \coloneqq \neg \univin \neg \phi \qquad \Diamondin \phi \coloneqq \neg \Boxin \neg \phi \qquad \Deltain{i} (\phi,\psi) \coloneqq \neg \nablain{i} (\neg \phi,\neg \psi).
\end{align*}
We extend the interpretation in Kripke frames to all the formulas of the reversive extension of $\mathcal{L}_{\nabla \Box}$ in the following way.
\begin{align*}
x \vDash_v \univin \phi \quad & \text{iff} \quad \forall y \in X \, (yEx \text{ implies } y \vDash_v \phi )\\
x \vDash_v \nablain{1}(\phi,\psi) \quad & \text{iff} \quad \forall y, z \in X \, (Tyxz \text{ implies } y \vDash_v \phi \text{ or } z \vDash_v \psi)\\
x \vDash_v \nablain{2}(\phi,\psi) \quad & \text{iff} \quad \forall y, z \in X \, (Tzyx \text{ implies } y \vDash_v \phi \text{ or } z \vDash_v \psi)\\
x \vDash_v \Boxin \phi \quad & \text{iff} \quad \forall y \in X \, (ySx \text{ implies } y \vDash_v \phi).
\end{align*}
The SQEMA algorithm will manipulate sets of formulas in the hybrid language obtained by adding \emph{nominals} to the reversive extension of $\mathcal{L}_{\nabla \Box}$. Nominals are a special sort of propositional variables and will be denoted by bold letters $\mathbf{i}, \mathbf{j}, \mathbf{k}, \dots$. We require that valuations map nominals to singletons. In this way, nominals serve as names for the elements of the frame.

Let $\phi$ be a formula in this hybrid language. The \emph{standard translation} $\text{ST}(\phi, x)$ of $\phi$ is a first-order formula in the first-order language containing the binary relation symbols $E,S$ and the ternary relation symbol $T$. The standard translation on the connectives of the reversive extension of $\mathcal{L}_{\nabla \Box}$ is defined in the usual way (see, e.g., \cite[Def.~2.45]{BRV01}) that reflects the interpretation of the connectives in Kripke frames given above. If $\mathbf{j}$ is a nominal, then $\text{ST}(\mathbf{j}, x)$ is the formula $x=y_j$, where $y_j$ is a reserved variable associated to the nominal $\mathbf{j}$ (for more details on the standard translations of formulas in hybrid languages see, e.g., \cite[p.~585]{CoGoVa06b}).

The algorithm SQEMA, which is given in full detail in \cite[Sec.~3]{CoGoVa06b}, takes as input a modal formula $\phi$ and, if it succeeds, it outputs a first-order formula that is a local first-order correspondent of $\phi$. A local first-order correspondent of a formula $\phi$ in $\mathcal{L}_{\nabla \Box}$ is a formula $\alpha(x)$ in the first-order language with identity, two binary relation symbols $E,S$, and a ternary relation symbol $T$ with the following property: for every Kripke frame $(X,E,T,S)$ and $a \in X$, we have $a \vDash_v \phi$ for any valuation $v$ iff $\alpha(a)$ holds in the frame. In particular, $\phi$ is valid in the frame iff $\forall x\, \alpha(x)$ holds.
For the general definition of local first-order correspondent see, e.g., \cite[Def.~3.29]{BRV01}.

We now briefly describe the algorithm, assuming the input is a formula $\phi$ in $\mathcal{L}_{\nabla \Box}$.\\
\textbf{Phase 1.} The formula $\neg \phi$ is rewritten into an equivalent disjunction of formulas $\bigvee \alpha_k$ that does not contain the connectives $\to$ and $\biimpl$ and is such that $\neg$ only occurs in front of propositional variables and no further distribution of $\ex, \Diamond, \Delta$, and $\wedge$ over $\vee$ is possible.\\
\textbf{Phase 2.} The algorithm then manipulates sets of formulas that are called \emph{systems} and are denoted with double vertical bars on their left. Each disjunct $\alpha_k$ yields an initial system with a single formula $\lVert \neg \mathbf{i} \vee \alpha_k$, where $\mathbf{i}$ is a fixed, reserved nominal.
Some transformation rules will be applied to the formulas of the systems and the algorithm succeeds if it manages to eliminate all the variables from each system, otherwise it terminates reporting failure. The following are the rules that we will use in our case.\\
\textbf{Rules for the connectives:}
\begingroup
\addtolength{\jot}{1em}
\begin{gather*}
\text{($\wedge$-rule)} \qquad \inference{\phi \vee (\psi \wedge \chi)}{\phi \vee \psi, \, \phi \vee \chi}\\
\text{($\univ$-rule)} \qquad \inference{\phi \vee \univ \psi}{\univin \phi \vee \psi} \hspace{1.2cm} \text{($\Box$-rule)} \qquad \inference{\phi \vee \Box \psi}{\Boxin \phi \vee \psi}\\
\end{gather*}
\begin{gather*}
\text{($\nabla$-rules)} \qquad \inference{\phi \vee \nabla (\psi_1,\psi_2)}{\nablain{1}(\phi,\psi_2) \vee \psi_1} \qquad \inference{\phi \vee \nabla (\psi_1,\psi_2)}{\nablain{2}(\psi_1,\phi) \vee \psi_2}\\
\text{($\ex$-rule)} \qquad \inference{\neg \mathbf{j} \vee \ex \psi}{\neg \mathbf{j} \vee \ex \mathbf{k},\, \neg \mathbf{k} \vee \psi} \hspace{1.2cm} \text{($\Diamond$-rule)} \qquad \inference{\neg \mathbf{j} \vee \Diamond \psi}{\neg \mathbf{j} \vee \Diamond \mathbf{k},\, \neg \mathbf{k} \vee \psi}\\
\text{($\Delta$-rule)} \qquad \inference{\neg \mathbf{j} \vee \Delta (\psi_1,\psi_2)}{\neg \mathbf{j} \vee \Delta (\mathbf{k}_1,\mathbf{k}_2),\, \neg \mathbf{k}_1 \vee \psi_1,\, \neg \mathbf{k}_2 \vee \psi_2},
\end{gather*}
\endgroup
where $\mathbf{j}, \mathbf{k}, \mathbf{k}_1, \mathbf{k}_2$ are nominals and $\mathbf{k}, \mathbf{k}_1, \mathbf{k}_2$ do not appear in the premises.\\
\textbf{Ackermann rule:} If $p$ is a variable that does not occur in the formulas $\phi_1, \dots, \phi_n$ and each of the formulas $\psi_1, \dots, \psi_m$ is negative in $p$ or does not contain $p$, then the rule replaces\\[1ex]
\begin{minipage}{0.4\textwidth}
\hspace{0.5cm} a system \hspace{1cm}
$
\begin{syst}
\phi_1 \vee p\\
\hspace{0.6cm}\vdots\\
\phi_n \vee p\\
\psi_1(p)\\
\hspace{0.4cm}\vdots\\
\psi_m(p)
\end{syst}
$
\end{minipage}%
\begin{minipage}{0.5\textwidth}
with \hspace{1cm}
$
\begin{syst}
\psi_1((\phi_1 \wedge \dots \wedge \phi_n)/\neg p)\\
\hspace{1.6cm}\vdots\\
\psi_m((\phi_1 \wedge \dots \wedge \phi_n)/\neg p)
\end{syst}
$
\end{minipage}\\[1ex]
\textbf{Polarity-switching rule:} 
If $p$ is a variable, then every occurrence of $\neg p$ is replaced with $p$ and every occurrence of $p$ not prefixed by $\neg$ is replaced with $\neg p$.\\
\textbf{Phase 3.} If the algorithm succeeds, then each system is rewritten into a system consisting of formulas that do not contain propositional variables. Let $\mathsf{pure}_k$ be the conjunction of the formulas in the $k$-th system, and define $\mathsf{pure}(\phi) \coloneqq \bigvee \mathsf{pure}_k$. Let $\overline{y}$ be the tuple of reserved variables that are associated to the nominals occurring in $\mathsf{pure}(\phi)$ except for the special nominal $\mathbf{i}$, which corresponds to a reserved variable $x$ that is not in $\overline{y}$. The algorithm returns the first-order formula $\forall \overline{y}\, \exists x_0 \, \text{ST}(\neg \mathsf{pure}(\phi), x_0)$, in which the only variable that occurs free is $x$.

\begin{theorem}{\textup{\cite[Thms.~4.3, 5.14]{CoGoVa06b}}}
If SQEMA succeeds on $\phi$, then $\phi$ is canonical and the output of the algorithm is a local first-order correspondent of $\phi$. 
\end{theorem}

We are now ready to employ the SQEMA algorithm to prove the following theorem that immediately implies \cref{prop:L Kripke complete}.

\begin{theorem}\label{thm:appendix}
All the axioms of $\MSSIC$ are canonical and
\begin{enumerate}
\item\label{thm:appendix:item1} \textup{(A0)} locally corresponds to $\forall z \, (xEz \biimpl \exists y \,Txyz)$;
\item\label{thm:appendix:item2} \textup{(A1)}, \textup{(A2)}, and \textup{(A3)} locally correspond to $\top$;
\item\label{thm:appendix:item3} \textup{(A4)} locally corresponds to $Txxx$;
\item\label{thm:appendix:item4} \textup{(A5)} locally corresponds to $\forall y, z \, (Txyz \to Txzy)$;
\item\label{thm:appendix:item5} \textup{(A8)} locally corresponds to $\forall y,z \,((x E y \wedge y E z) \to xEz)$ and is a consequence of \textup{(A5)} and \textup{(A9)};
\item\label{thm:appendix:item6} \textup{(A9)} locally corresponds to $\forall y,z \, ((xEy \wedge xEz) \to yEz)$;
\item\label{thm:appendix:item7} The right-to-left implication of \textup{(A10)} is a consequence of \textup{(A5)};
\item\label{thm:appendix:item8} The left-to-right implication of \textup{(A10)} locally corresponds to 
\[
\forall y,z,w \, ((x E w \wedge Twyz) \to Txyz);
\]
\item\label{thm:appendix:item9} \textup{(A11)} is a consequence of \textup{(A0)}, \textup{(A5)}, and \textup{(A8)};
\item\label{thm:appendix:item10} \textup{(K)} locally corresponds to $\top$;
\item\label{thm:appendix:item11} \textup{(Add)} locally corresponds to 
\[
\forall y,z,w \, ((Txyz \wedge z S w) \to \exists u \,(Txuw \wedge y S u)).
\]
\end{enumerate}
\end{theorem}

\begin{proof}
\eqref{thm:appendix:item1}.
We execute the SQEMA algorithm with the axiom $\univ p \biimpl \nabla(\bot,p)$ as input.
We first negate the formula and rewrite it to obtain 
\[
(\univ p \wedge \Delta(\top,\neg p)) \vee (\nabla(\bot,p) \wedge \ex (\neg p)).
\]
The two disjuncts give two initial systems
\[
\lVert \neg \mathbf{i} \vee (\univ p \wedge \Delta(\top,\neg p)) \quad \text{and} \quad \lVert \neg \mathbf{i} \vee (\nabla(\bot,p) \wedge \ex (\neg p)).
\]
Applying the $\wedge$-rule and the $\univ$-rule to the first system gives
\[
\begin{syst}
\univin(\neg \mathbf{i}) \vee p \\ 
\neg \mathbf{i} \vee \Delta(\top,\neg p).
\end{syst}
\]
Then the Ackermann rule eliminates $p$ from the system and yields
\[
\begin{syst}
\neg \mathbf{i} \vee \Delta(\top,\univin(\neg \mathbf{i})).
\end{syst}
\]
We now turn our attention to the second system and apply the $\wedge$-rule and the $\nabla$-rule. So, we obtain
\[
\begin{syst}
\nablain{2}(\bot,\neg \mathbf{i}) \vee p\\
\neg \mathbf{i} \vee \ex(\neg p).
\end{syst}
\]
We now apply the Ackermann rule to eliminate $p$:
\[
\begin{syst}
\neg \mathbf{i} \vee \ex\nablain{2}(\bot,\neg \mathbf{i}).
\end{syst}
\]
Thus, the algorithm succeeds and guarantees that (A0) is canonical. The negation of the disjunction of the formulas in the two systems is equivalent to the following formula
\[
\mathbf{i} \wedge \nabla(\bot,\exin(\mathbf{i})) \wedge \univ\Deltain{2}(\top,\mathbf{i}),
\]
whose corresponding first-order formula is equivalent to 
\[
\forall z \, (xEz \biimpl \exists y \,Txyz).
\]

\eqref{thm:appendix:item2}.
Axioms (A1), (A2), (A3) express the normality of the modality $\nabla$, so they are clearly canonical and locally correspond to $\top$.

\eqref{thm:appendix:item3}.
We execute the SQEMA algorithm with the axiom $\nabla(p,q) \impl (p \vee q)$ given as input. We negate the formula and rewrite it as follows 
\[
\nabla(p,q) \wedge \neg p \wedge \neg q.
\]
The formula gives a single system
\[
\begin{syst}
\neg \mathbf{i} \vee (\nabla(p,q) \wedge \neg p \wedge \neg q).
\end{syst}
\]
Applying the $\wedge$-rule and the polarity-switching rule on both $p$ and $q$ yields
\[
\begin{syst}
\neg \mathbf{i} \vee \nabla(\neg p, \neg q)\\
\neg \mathbf{i} \vee p\\
\neg \mathbf{i} \vee q.
\end{syst}
\]
The Ackermann rule can be used twice to eliminate both $p$ and $q$:
\[
\begin{syst}
\neg \mathbf{i} \vee \nabla(\neg \mathbf{i}, \neg \mathbf{i}).
\end{syst}
\]
Thus, the algorithm succeeds on (A4), which is then canonical. The negation of the only formula in the system is equivalent to
\[
\mathbf{i} \wedge \Delta(\mathbf{i},\mathbf{i}),
\]
whose corresponding first-order formula is equivalent to 
\[
Txxx.
\]

\eqref{thm:appendix:item4}.
Since (A5) is the conjunction of $\nabla(p,q) \impl \nabla(q,p)$ and $\nabla(q,p) \impl \nabla(p,q)$, it is sufficient to run SQEMA on the formula $\nabla(p,q) \impl \nabla(q,p)$. Its negation can be rewritten as
\[
\nabla(p,q) \wedge \Delta(\neg q,\neg p).
\]
We obtain a single system
\[
\begin{syst}
\neg \mathbf{i} \vee (\nabla(p,q) \wedge \Delta(\neg q,\neg p)).
\end{syst}
\]
The $\wedge$-rule and the $\nabla$-rule give
\[
\begin{syst}
\nablain{1}(\neg \mathbf{i},q) \vee p\\
\neg \mathbf{i} \vee \Delta(\neg q,\neg p).
\end{syst}
\]
We eliminate $p$ using the Ackermann rule:
\[
\begin{syst}
\neg \mathbf{i} \vee \Delta(\neg q,\nablain{1}(\neg \mathbf{i},q)).
\end{syst}
\]
We then use the $\Delta$-rule and the polarity-switching rule on $q$ to obtain
\[
\begin{syst}
\neg \mathbf{i} \vee \Delta(\mathbf{j}_1,\mathbf{j}_2)\\
\neg \mathbf{j}_1 \vee q\\
\neg \mathbf{j}_2 \vee \nablain{1}(\neg \mathbf{i},\neg q).
\end{syst}
\]
The Ackermann rule eliminates $q$:
\[
\begin{syst}
\neg \mathbf{i} \vee \Delta(\mathbf{j}_1,\mathbf{j}_2)\\
\neg \mathbf{j}_2 \vee \nablain{1}(\neg \mathbf{i},\neg \mathbf{j}_1).
\end{syst}
\]
Thus, SQEMA succeeds on the formula and guarantees its canonicity. The negation of the conjunction of the two formulas in the system is equivalent to
\[
(\mathbf{i} \wedge \neg\Delta(\mathbf{j}_1,\mathbf{j}_2)) \vee (\mathbf{j}_2 \wedge \Deltain{1}(\mathbf{i},\mathbf{j}_1)),
\]
whose corresponding first-order formula is equivalent to
\[
\forall y, z \, (Txyz \to Txzy).
\]

\eqref{thm:appendix:item5}.
Axiom (A8) coincides with the axiom (K4) for $\univ$. It is well known that (K4) is canonical and locally corresponds to the first-order formula 
\[
\forall y,z ((x E y \wedge y E z) \to xEz).
\] 
We show that (A8) follows from (A5) and (A9). Axiom (A5) yields that $\MSSIC$ proves $\nabla(\bot,p) \to p$, and hence also $\univ p \to p$ by (A0). Note that $\univ p \to p$ is the axiom (T) for $\univ$. As we will observe in \eqref{thm:appendix:item6}, (A9) is equivalent to the axiom (S5) for $\univ$. It is well known that (S5) together with (T) implies (K4). This shows that (A8) follows from (A5) and (A9).

\eqref{thm:appendix:item6}.
Axiom (A9) is equivalent to the (S5) axiom for $\univ$. It is well known that it locally correspond to $\forall y,z \, ((xEy \wedge xEz) \to yEz)$ and is canonical.

\eqref{thm:appendix:item7}.
As shown in \eqref{thm:appendix:item5}, $\univ p \to p$ is a consequence of (A0) and (A5). Thus, $\MSSIC$ proves $\univ \nabla(p,q) \to \nabla(p,q)$, which is the right-to-left implication of (A10).

\eqref{thm:appendix:item8}.
We run SQEMA on the left-to-right implication of (A10). We negate $\nabla(p,q) \impl \univ \nabla(p,q)$ and rewrite it:
\[
\nabla(p,q) \wedge \ex \Delta(\neg p,\neg q).
\]
We then get the system
\[
\begin{syst}
\neg \mathbf{i} \vee (\nabla(p,q) \wedge \ex \Delta(\neg p,\neg q)).
\end{syst}
\]
Using the $\wedge$-rule and the $\nabla$-rule we obtain
\[
\begin{syst}
\nablain{2}(p,\neg \mathbf{i}) \vee q\\
\neg \mathbf{i} \vee  \ex \Delta(\neg p,\neg q).
\end{syst}
\]
The Ackermann rule eliminates $q$ and yields
\[
\begin{syst}
\neg \mathbf{i} \vee \ex \Delta(\neg p,\nablain{2}(p,\neg \mathbf{i})).
\end{syst}
\]
We then use the $\ex$-rule and the $\Delta$-rule, and then we apply the polarity-switching rule on $p$:
\[
\begin{syst}
\neg \mathbf{i} \vee \ex \mathbf{j}_1\\
\neg \mathbf{j}_1 \vee \Delta(\mathbf{j}_2, \mathbf{j}_3)\\
\neg \mathbf{j}_2 \vee p\\
\neg \mathbf{j}_3 \vee \nablain{2}(\neg p,\neg \mathbf{i}).
\end{syst}
\]
The Ackermann rule allows to eliminate $p$:
\[
\begin{syst}
\neg \mathbf{i} \vee \ex \mathbf{j}_1\\
\neg \mathbf{j}_1 \vee \Delta(\mathbf{j}_2, \mathbf{j}_3)\\
\neg \mathbf{j}_3 \vee \nablain{2}(\neg \mathbf{j}_2,\neg \mathbf{i}).
\end{syst}
\]
Thus, SQEMA succeeds on the formula, which is then canonical. The negation of the conjunction of the formulas in the system is equivalent to the formula
\[
(\mathbf{i} \wedge \neg \ex \mathbf{j}_1) \vee (\mathbf{j}_1 \wedge \neg\Delta(\mathbf{j}_2,\mathbf{j}_3)) \vee (\mathbf{j}_3 \wedge \Deltain{2}(\mathbf{j}_2,\mathbf{i})),
\]
whose corresponding first-order formula is equivalent to
\[
\forall y,z,w \, ((x E w \wedge Twyz) \to Txyz).
\]

\eqref{thm:appendix:item9}.
It follows from (A8) and (A0) that $\univ p \to \nabla(\bot, \univ p)$ is a theorem of $\MSSIC$. Then (A5) yields $\univ p \to \nabla(\univ p, \bot)$, which is Axiom (A11).

\eqref{thm:appendix:item10}.
This is clear.

\eqref{thm:appendix:item11}.
We execute SQEMA on Axiom (Add). We negate $\nabla(p,q) \impl \nabla(\Diamond p, \Box q)$ and rewrite it as follows:
\[
\nabla(p,q) \wedge \Delta(\Box (\neg p), \Diamond (\neg q)).
\]
So, we get the system
\[
\begin{syst}
\neg \mathbf{i} \vee (\nabla(p,q) \wedge \Delta(\Box (\neg p), \Diamond (\neg q))).
\end{syst}
\]
By the $\wedge$-rule and the $\Delta$-rule, we obtain
\[
\begin{syst}
\neg \mathbf{i} \vee \nabla(p,q)\\
\neg \mathbf{i} \vee \Delta(\mathbf{j}_1,\mathbf{j}_2)\\
\neg \mathbf{j}_1 \vee \Box (\neg p)\\
\neg \mathbf{j}_2 \vee \Diamond (\neg q).
\end{syst}
\]
We apply the $\Box$-rule, the $\Diamond$-rule, and then the polarity-switching rule on both $p$ and $q$:
\[
\begin{syst}
\neg \mathbf{i} \vee \nabla(\neg p,\neg q)\\
\neg \mathbf{i} \vee \Delta(\mathbf{j}_1,\mathbf{j}_2)\\
\Boxin(\neg \mathbf{j}_1) \vee p\\
\neg \mathbf{j}_2 \vee \Diamond \mathbf{j}_3\\
\neg \mathbf{j}_3 \vee q.
\end{syst}
\]
Two applications of the Ackermann rule eliminate both variables $p$ and $q$
\[
\begin{syst}
\neg \mathbf{i} \vee \nabla(\Boxin(\neg \mathbf{j}_1),\neg \mathbf{j}_3)\\
\neg \mathbf{i} \vee \Delta(\mathbf{j}_1,\mathbf{j}_2)\\
\neg \mathbf{j}_2 \vee \Diamond \mathbf{j}_3.
\end{syst}
\]
Thus, SQEMA succeeds and the axiom is canonical. The negation of the conjunction of the formulas in the system is equivalent to
\[
(\mathbf{i} \wedge \neg\Delta(\mathbf{j}_1,\mathbf{j}_2)) \vee (\mathbf{j}_2 \wedge \neg\Diamond \mathbf{j}_3) \vee (\mathbf{i} \wedge \Delta(\Diamondin \mathbf{j}_1,\mathbf{j}_3)),
\]
whose corresponding first-order formula is equivalent to 
\[
\forall y,z,w \, ((Txyz \wedge z S w) \to \exists u \,(Txuw \wedge y S u)).\qedhere
\]
\end{proof}

\subsection*{Acknowledgements}

We would like to thank the referees for careful reading and useful comments which have improved the presentation of the paper.

Luca Carai and Silvio Ghilardi are members of the Gruppo Nazionale per le Strutture Algebriche, Geometriche e le loro Applicazioni (GNSAGA) of the Istituto Nazionale di Alta Matematica (INdAM) and gratefully acknowledge its support. The first three named authors acknowledge the support of the MSCA-RISE-Marie Skłodowska-Curie Research and Innovation Staff Exchange (RISE) project MOSAIC 101007627 funded by Horizon 2020 of the European Union. The research of Zhiguang Zhao is supported by Shandong Provincial Natural Science Foundation, China (project number: ZR2023QF021) and Taishan Young Scholars Program of the Government of Shandong Province, China (No.tsqn201909151).

\end{document}